\documentclass{article}
\usepackage{amsthm, amssymb, amsmath, url,verbatim}
\usepackage{graphicx}
\newtheorem{thm}{Theorem}
\newtheorem{prop}[thm]{Proposition}

\newtheorem{lem}[thm]{Lemma}

\newtheorem{claim}{Claim}

\newtheorem*{defn}{Definition}

\newcommand{\floor}[1]{\left\lfloor #1 \right\rfloor}
\newcommand{\Z}{\mathbb{Z}}

\title{Asymptotic Results for the Queen Packing Problem}
\author{Daniel M. Kane}

\begin{document}

\maketitle

\section{Introduction}

A classic chess problem is that of placing 8 queens on a standard board so that no two attack each other. This problem and its generalizations to larger board sizes are fairly well studied. In this paper, we will study the related problem of producing two distinct armies of queens so that no queen in one army is attacking any queen in the other army.

Variations of this problem have been considered since as early as 1896 in Rouse Ball's \emph{Mathematical Recreations and Essays}. In this book, it is asked what the largest number of unattacked squares that can be left by an army of 8 queens on a standard chess board. Since then this problem has shown up now and then (for example a variant shows up in problem 316 in \cite{dudeney}). More recently, there has been more serious interest in the problem. Work has been made on both a theoretical level and with the help of computer assistance (see for example \cite{smith}). For a brief survey of work on this problem see the article by Martin Gardner \cite{gardner}.

There seem to be two reasonable asymptotic limits, one in which the size of the board, $n$, is fixed and the objective is to find two non-attacking armies of size $m$ for $m$ as large as possible. This problem appears on the Online Encyclopedia of Integer Sequences (\cite{oeis}). Despite much study little is still known about the general answer. The optimal configuration for an $11\times 11$ board was shown to be $m=17$ by \cite{smith}, and the configuration is shown in Figure \ref{queenSquareFig}. A natural scaling of this type of configuration gives an asymptotic lower bound of $m\geq (9/64)n^2+O(n)$ for larger boards, but very little is known beyond this.

The other natural asymptotic regime is to fix the size $k$ of one of the armies and to let the size $n$ of the board go to infinity. It is clear that in this regime, we want to place our $k$ queens so that they attack as few total squares as possible so that the other army can occupy the remaining squares. It is clear that for fixed $k$, the minimum number of attacked squares is going to be proportional to $n$. Somehow the real question is how few rows, columns and diagonals can you get away with attacking. In this paper, we pin down the constant of proportionality as a function of $k$.

When considering this, problem the first thing to note is that you do not have much control over how many squares a \emph{given} queen attacks. The number will vary from approximately $3n$ squares for queens near the edge of the board to roughly $4n$ for queens near the center, but because each queen attacks a full row and a full column (in addition to diagonals), each will attack some multiple of $n$ squares.

\begin{figure}
\begin{center}
\includegraphics[scale=0.2]{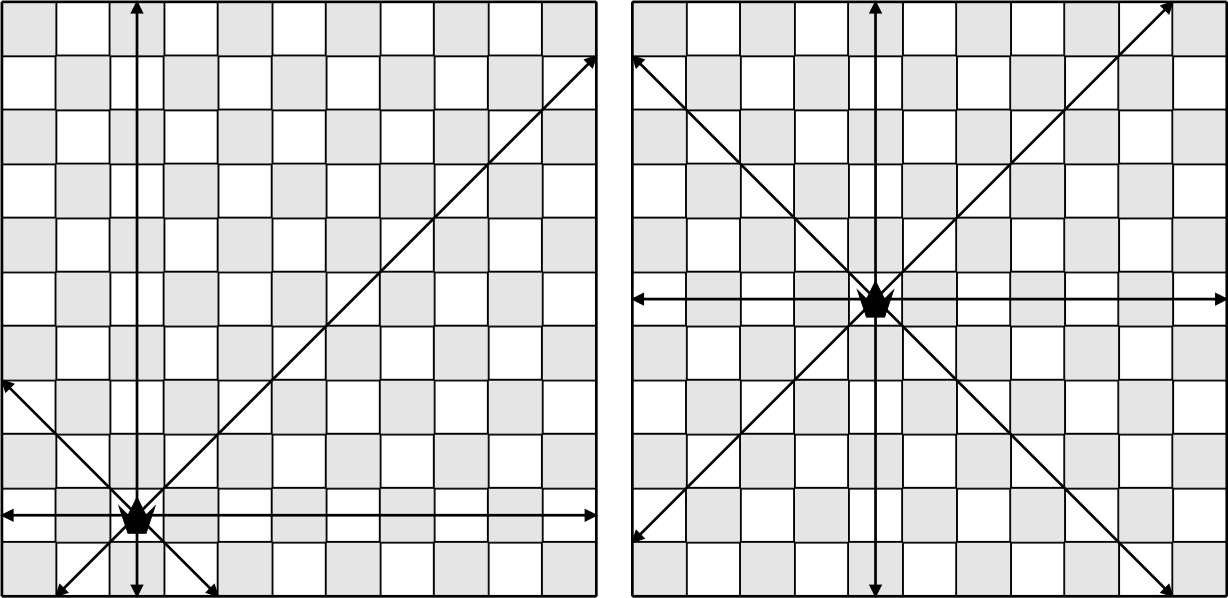}
\caption{The squares defended by individual queens. Ones near the center of the board defend more squares.}\label{singleQueensFig}
\end{center}
\end{figure}

If each queen in one army attacks a different $n$ squares, you will very quickly cover the entire board. However, this can be avoided by having multiple queens attacking the same square. Any two queens will attack some squares in common, however for most pairs of queens this cannot be more than a small number. In particular, if one picks for each of the two queens one of the four directions that it can attack along, you will typically find at most one square of intersection for these lines. This means that most pairs of queens will attack at most 16 squares (actually 12 if you note that the parallel lines will not intersect) in common. This will always hold unless the queens share a row, column, or diagonal.

\begin{figure}
\begin{center}
\includegraphics[scale=0.2]{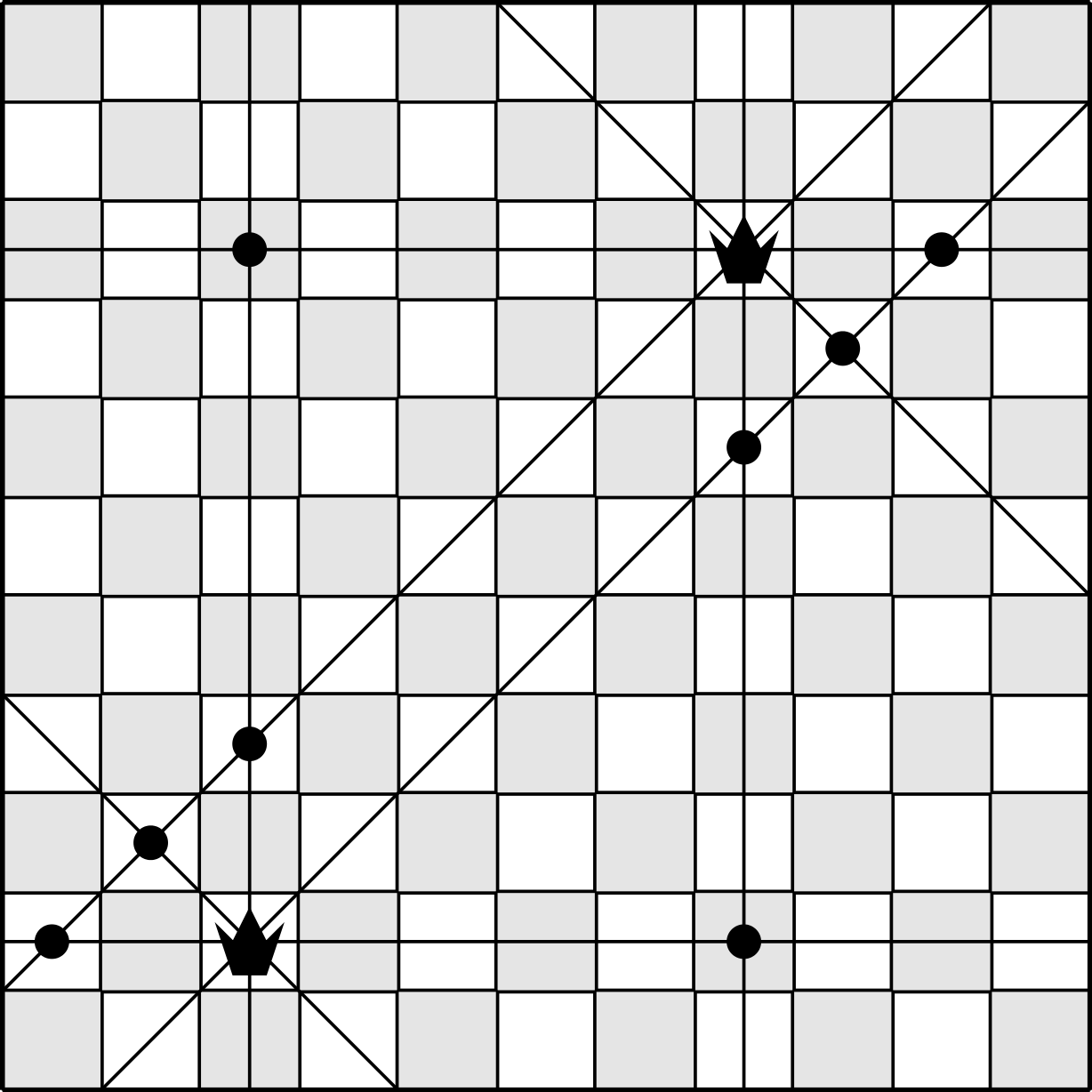}
\caption{The squares attacked in common by a pair of queens.}\label{sharedSquaresFig}
\end{center}
\end{figure}

This suggests that the correct strategy is to pack the queens densely together into regions. For example if the queens are placed in a $k\times k$ square as shown in Figure \ref{queenSquareFig}. In this way you can place $k^2$ queens, but only attack squares on one of $k$ rows, $k$ columns and $4k$ diagonals. Therefore with $k^2$ queens, you attack a total of at most $6kn$ squares, which scales like the square root of the number of queens rather than linearly.

\begin{figure}
\begin{center}
\hfill \includegraphics[scale=0.1]{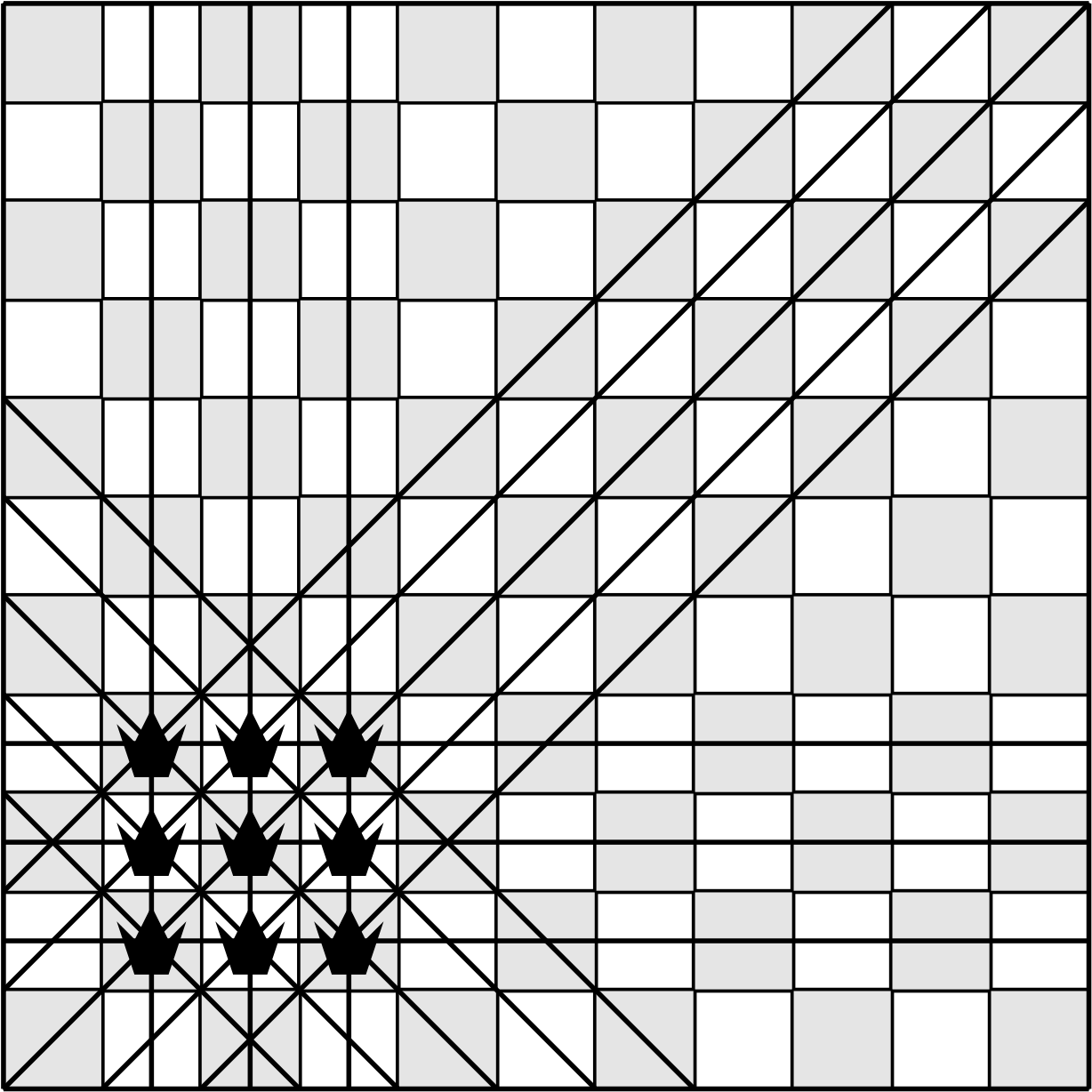}\hfill\includegraphics[scale=0.1]{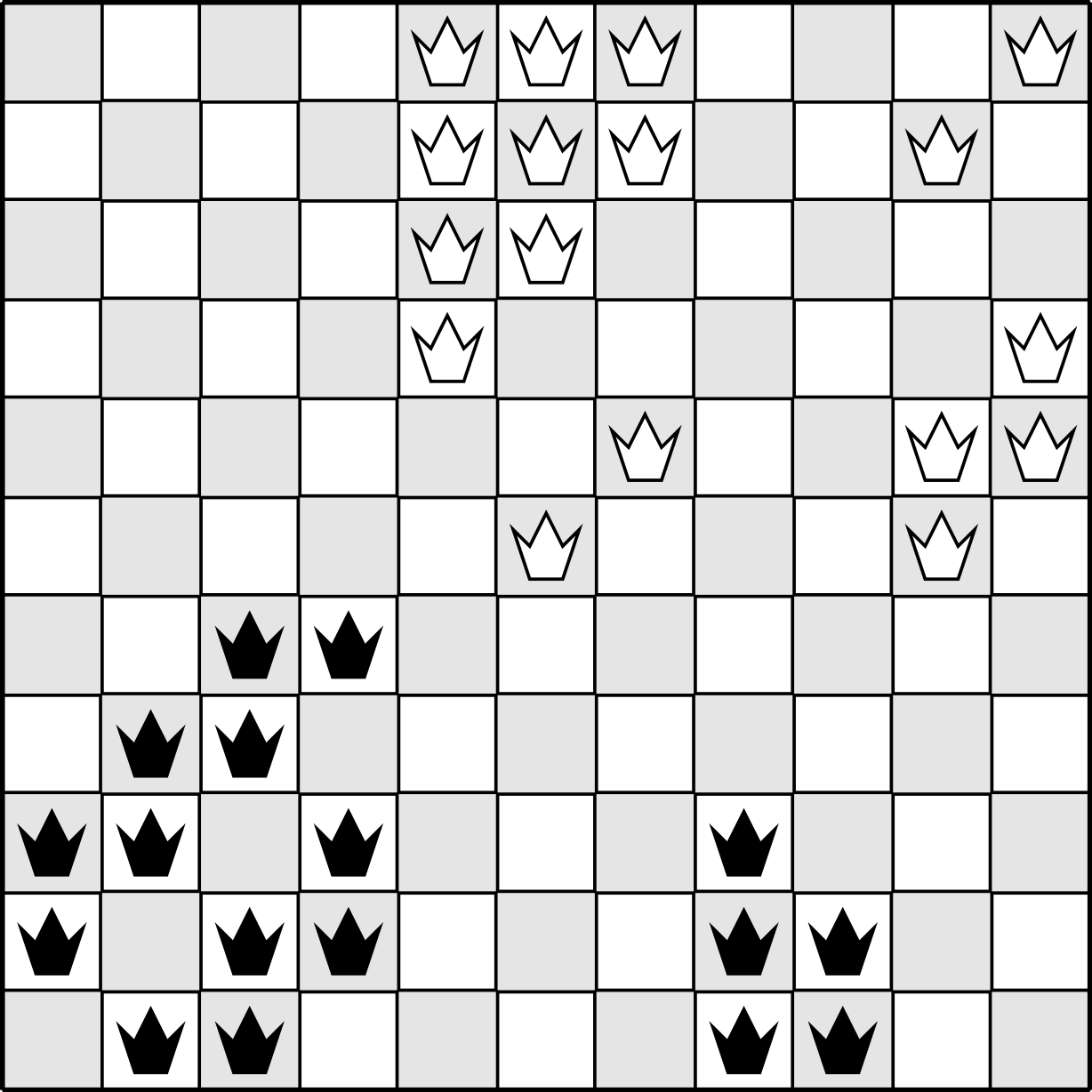}\hfill \phantom{.}
\caption{Left: The squares attacked by a $3\times 3$ block of queens. Right: Two non-attacking queen armies of size $17$ on an $11\times 11$ board.}\label{queenSquareFig}
\end{center}
\end{figure}

\section{Good Arrangements}

An obvious attempt is to arrange the queens in an $m\times m$ square. Unfortunately, the queens on the corners here will attack new diagonals all by themselves, removing these queens will be a very cheap way to reduce the number of attacked squares. If your queens are all located on a corner of the board, only some of these diagonals matter, but by shaving off the appropriate corners of the square, you end up with a much more efficient hexagon as seen in figure \ref{arrangementsFig}.

\begin{figure}
\begin{center}
\hfill \includegraphics[scale=0.1]{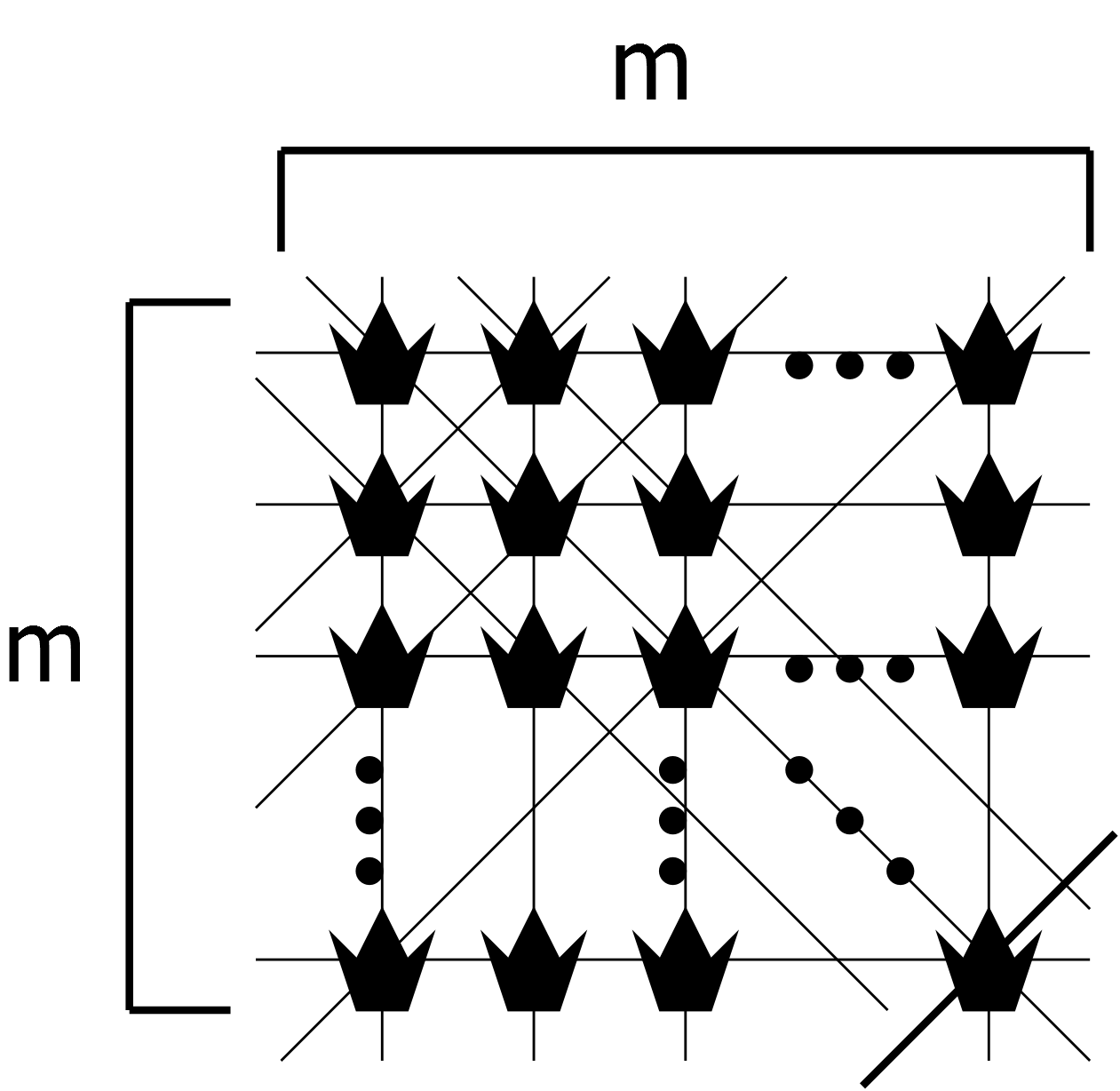}\hfill\includegraphics[scale=0.1]{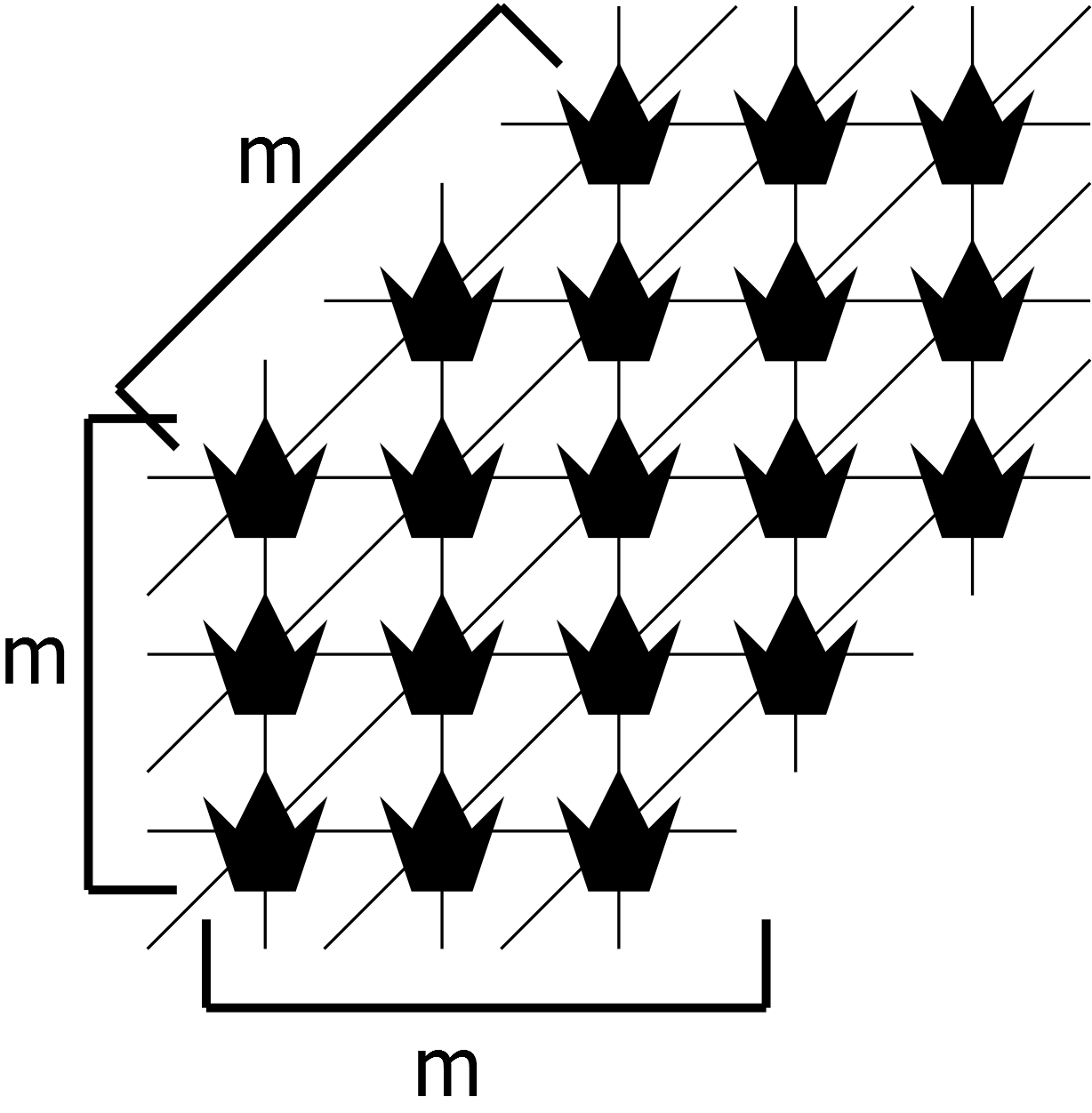}\hfill \phantom{.}
\caption{Left: An $m\times m$ square of queens. Right: A hexagon of queens.}\label{arrangementsFig}
\end{center}
\end{figure}

The queens in the hexagon shown cover $2m-1$ rows, $2m-1$ columns, and $2m-1$ long diagonals (assuming that the figure is located in the bottom left corner of the board). Thus, there are a total of approximately $(6m-3)n$ attacked squares. It is not hard to see that the total number of queens in this region is $3m^2-3m+1$. This works very well when the number of queens is exactly $3m^2-3m+1$ for some integer $m$, but for intermediate numbers of queens, you might want to use a hexagon with slightly unequal sides in order to fit them in.

\begin{figure}
\begin{center}
\hfill\includegraphics[scale=0.1]{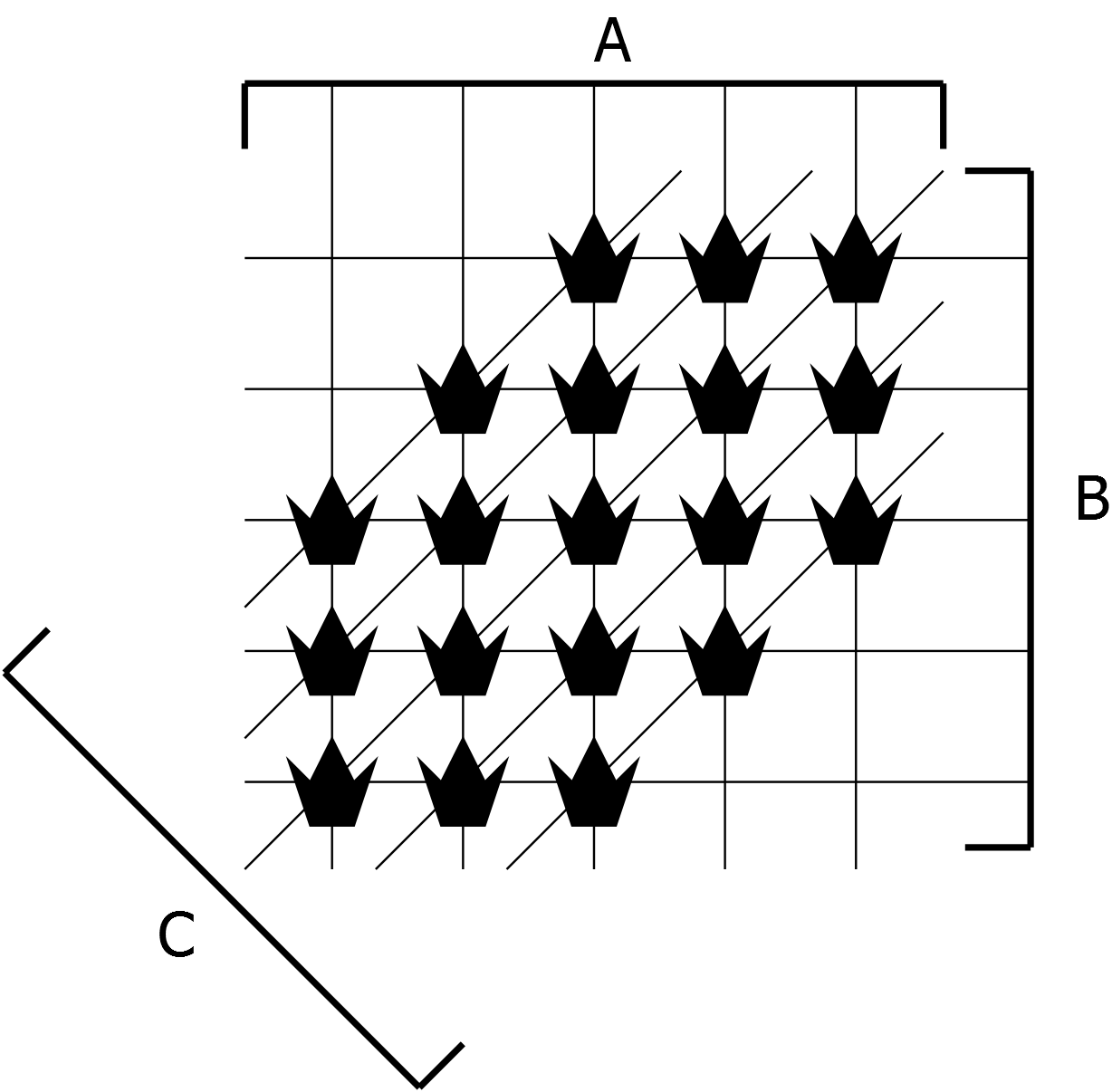}\hfill\includegraphics[scale=0.1]{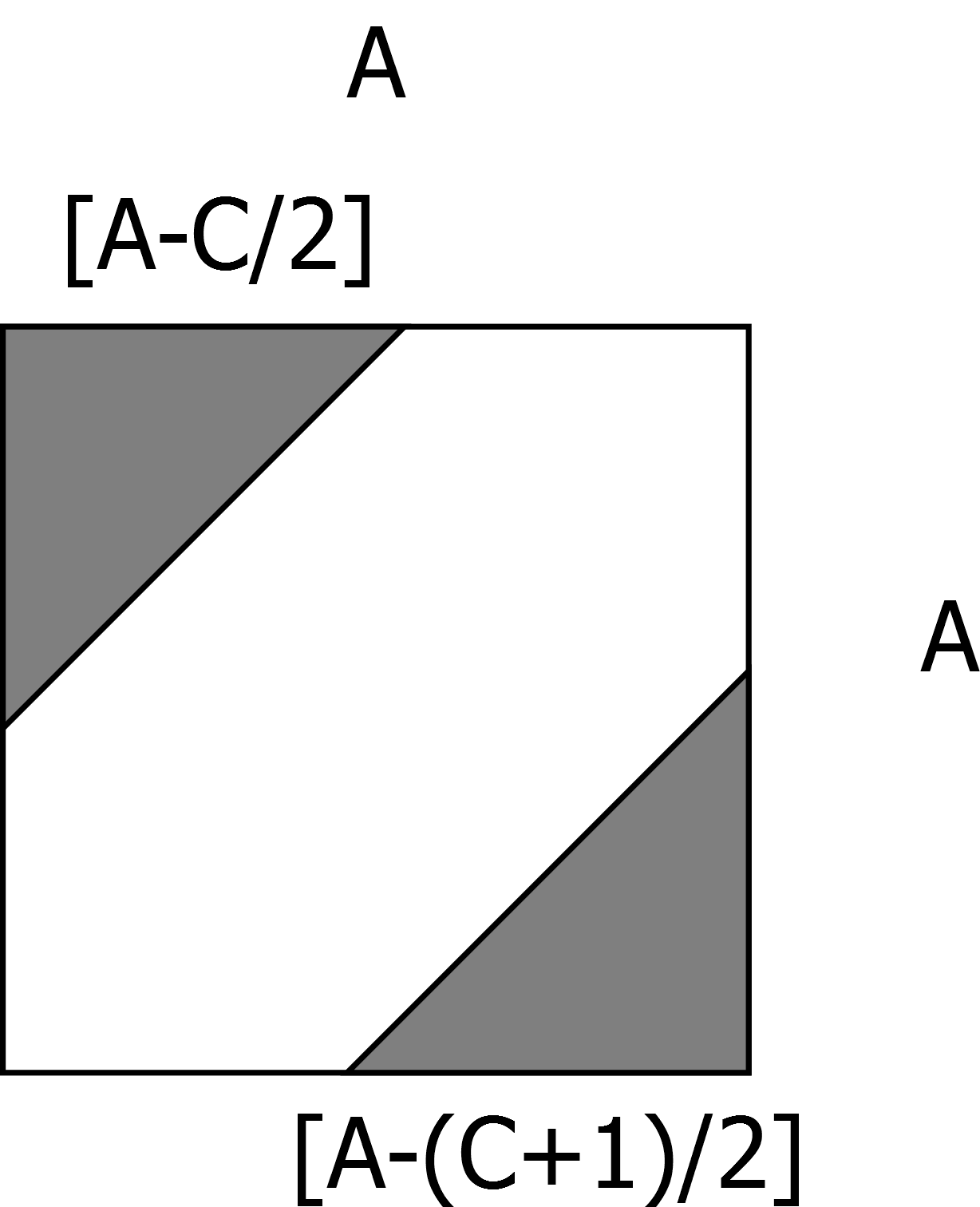} \hfill \phantom{.}
\caption{Left: An uneven hexagon. Right: Computing the number of queens.}\label{genHexFig}
\end{center}
\end{figure}

In particular, if we wish to defend approximately $mn$ squares, we can consider a hexagon consisting of all the queens on the leftmost $A$ columns, the bottommost $B$ rows, and the middle $C$ diagonals for the unique triple of integers $A,B,C$ with $A+B+C=m$, $A=B$ and $|A-C|\leq 1$. It is clear that the resulting arrangement of queens attacks at most $mn$ squares, the question is how many queens are in the arrangement. In order to answer this question, we note that there are $A^2$ squares at the intersection of a specified row and specified column, however of the $2A-1$ positive diagonals that these squares lie on, only $C$ of them are in the specified set. The uncovered squares consist of two right triangles of side lengths $\floor{(2A-C-1)/2}$ and $\floor{(2A-C)/2}$. Therefore, the total number, $K$, of squares at a three way intersection is
$$
A^2 - \floor{(2A-C-1)/2}(\floor{(2A-C-1)/2}+1)/2 - \floor{(2A-C)/2}(\floor{(2A-C)/2}+1)/2.
$$
A simple computation shows that this is always equal to $\floor{\frac{m^2+3}{12}}.$

One might, expect that this is always the optimal, but in the case where $m$ is congruent to $6$ modulo $12$, one can actually squeeze in an extra queen. In particular, instead of putting one hexagon in a corner with $A=B=C=m/3$ yielding $m^2/12$ queens, you can instead put a hexagon with $A=B=C=m/6$ in \emph{each} corner of the board without attacking any more squares. On the other hand, this allows you squeeze in an extra queen as shows in Figure \ref{fourHexFig}.

\begin{figure}
\begin{center}
\hfill\includegraphics[scale=0.1]{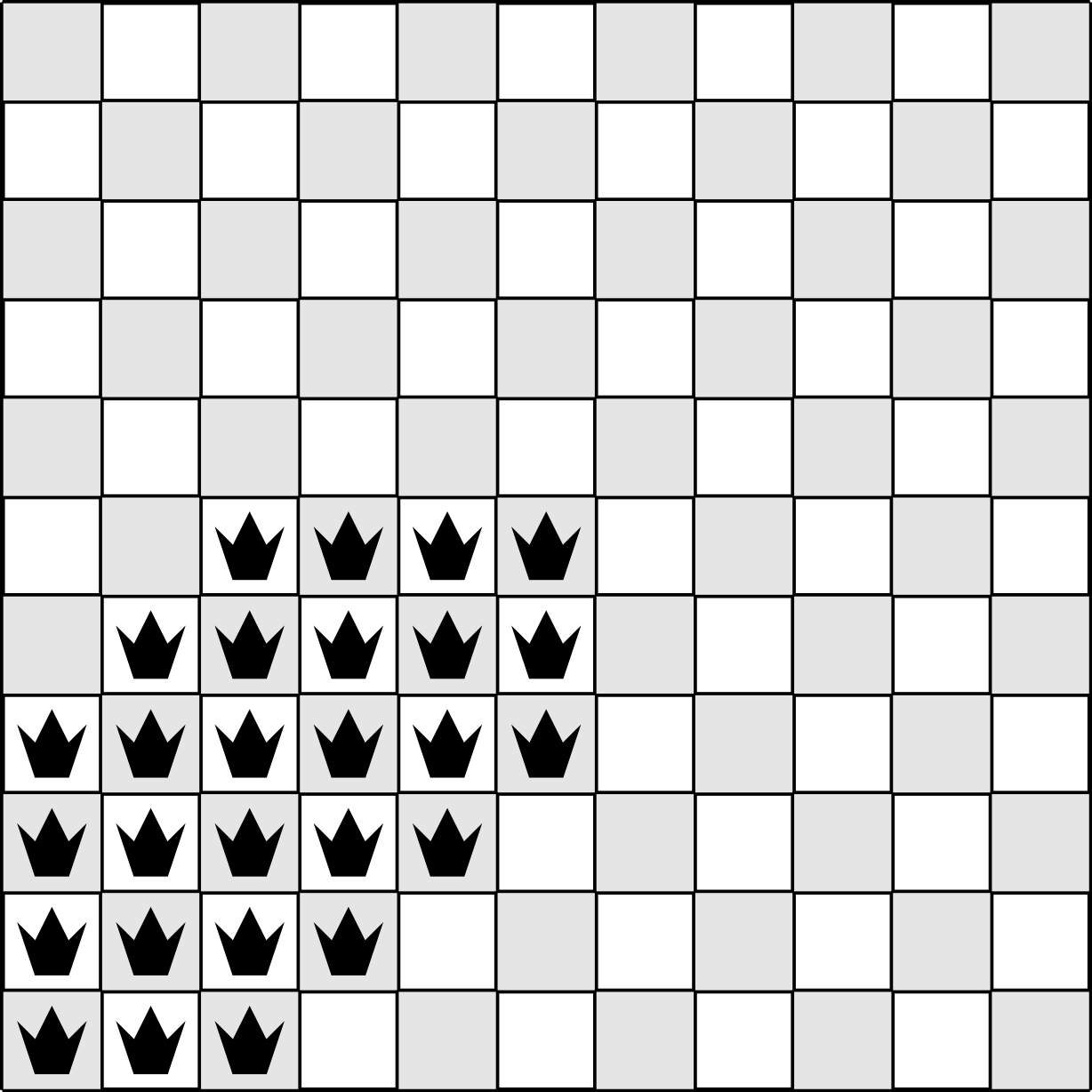}\hfill\includegraphics[scale=0.1]{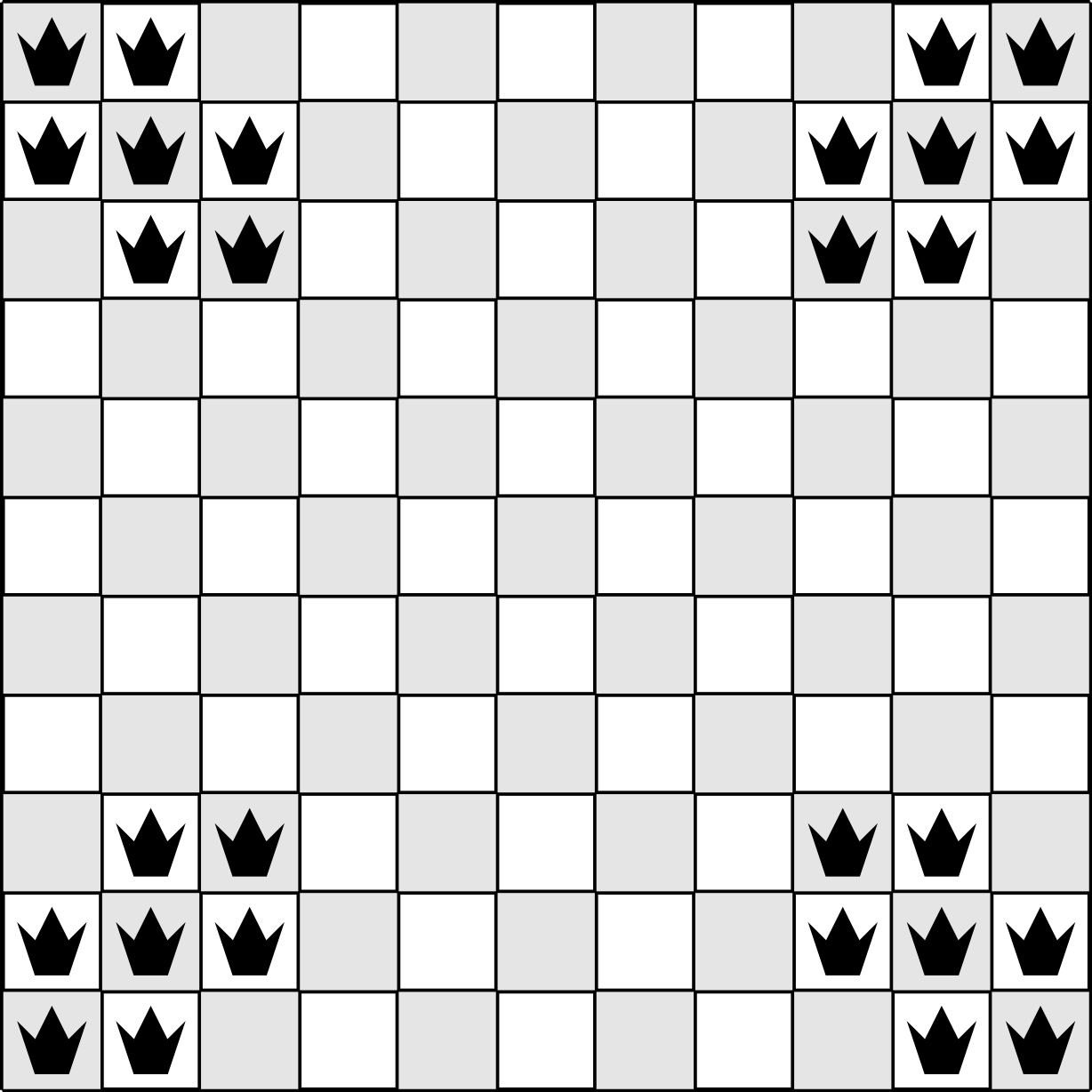} \hfill \phantom{.}
\caption{One hexagon with 6 rows, columns and long diagonals can only fit $27$ queens, but by putting one hexagon in each corner, you can manage $28$.}\label{fourHexFig}
\end{center}
\end{figure}

These configurations seem to be close to optimal in almost all cases when the number of queens is relatively small compared to the size of the board. However, there is one notable exception. If there are exactly $9$ queens on the board, one can put one queen in each corner, one in the middle of each side and one in the center (if $n$ is odd, for $n$ even, you need to slightly modify the construction) as shown below in Figure \ref{ninequeensFig}. When this happens, the queens attack squares on three rows, three columns, two long diagonals and four half-diagonals, for a total of roughly $10n$ squares. Note that the previous constructions could only do this with at most $8$ queens.

\begin{figure}
\begin{center}
\includegraphics[scale=0.1]{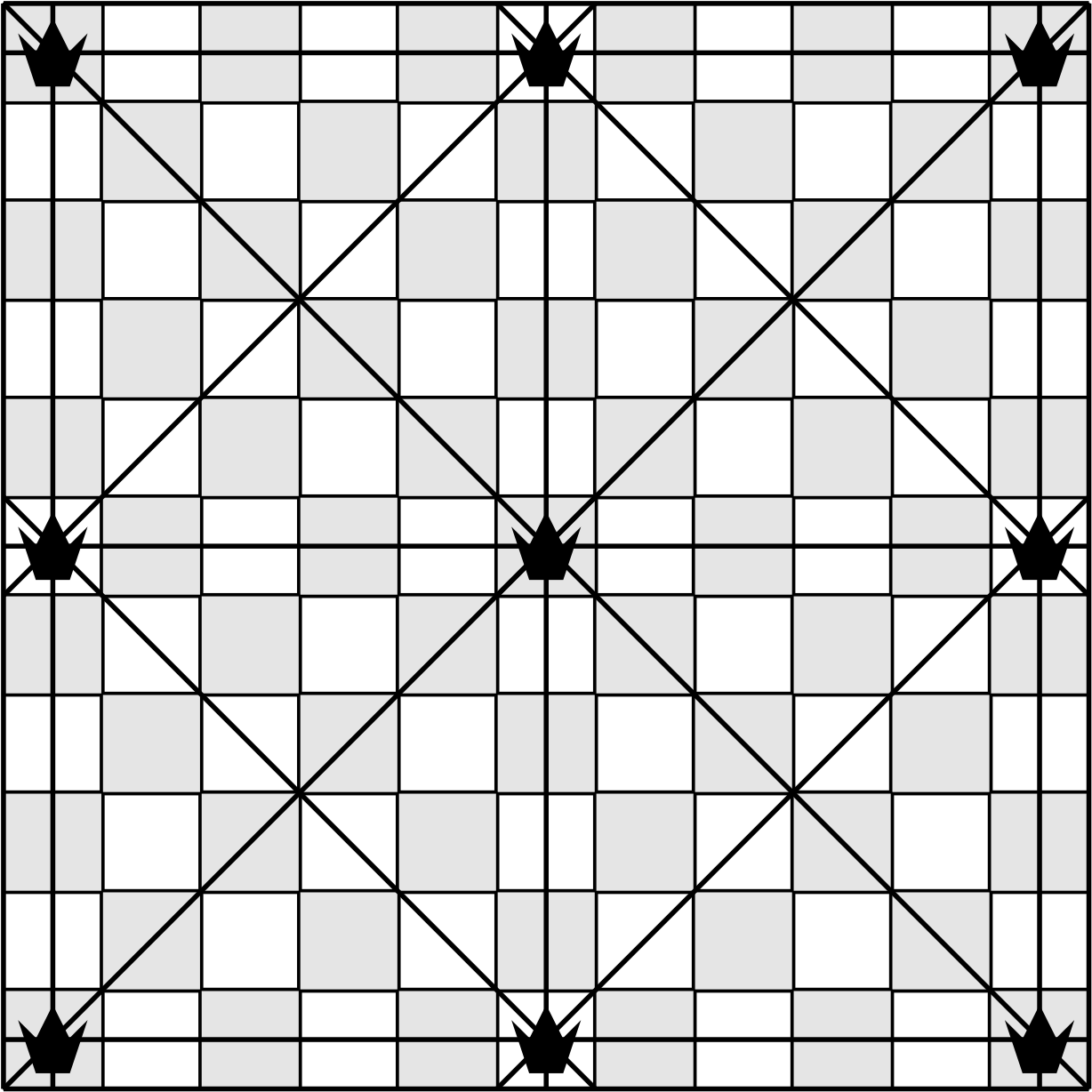}
\caption{A placement of nine queens on an  $11\times 11$ board.}\label{ninequeensFig}
\end{center}
\end{figure}

It turns out though, that this is the only remaining case. In order to summarize our constructions thus far, we define a function $G(m)$, the greatest number of queens that we can put on the board in this way without attacking more than about $nm$ squares. With a little bit of work, we can compute the function as follows:
\begin{defn}
$$
G(m) = \begin{cases}\floor{\frac{m^2}{12}}+1 & \textrm{if }m\equiv 3,6,9\pmod{12}, \textrm{ or if }m=10\\ & \\\floor{\frac{m^2}{12}} &\textrm{otherwise}. \end{cases}
$$
\end{defn}

Using this we can state our main Theorem:
\begin{thm}\label{mainThm}
Let $k$ and $n$ be positive integers with $k\leq n^2$. Let $m$ be the smallest positive integer so that $G(m)\geq k$. Then among placements of $k$ queens on an $n\times n$ board the minimum possible number of total squares attacked by these queens is
$
mn+O(k).
$
\end{thm}

Note that we have already proved the requisite upper bound, namely that $G(m)$ queens can be placed without attacking more than $mn$ squares using the constructions above. Proving a matching lower bound will be somewhat more challenging.

\section{Lower Bounds}

To prove the lower bound, we will need to know that if $k>G(m)$ that we will need to attack substantially more than $mn$ squares. To get this bound, we will consider the lengths of the rows, columns and diagonals that queens lie on. In particular, we will prove:

\begin{prop}\label{mainProp}
Suppose that $k$ queens are placed on an $n\times n$ chessboard for $k>G(m)$ for some integer $m$. Then there are $O(m)$ rows, columns and diagonals with queens on them so that the sum of the lengths of these rows, columns and diagonals is at least $n(m+1)$.
\end{prop}

Our lower bound will follow immediately from Proposition \ref{mainProp} by noting that each pair of two rows, columns or diagonals intersect in at most 1 point. Therefore the size of the union of the lines in Proposition \ref{mainProp} is the sum of their lengths minus $O(m^2)=O(k)$.

Before we proceed to the more complicated argument, we begin by looking at the simple case where all of the queens are located near the top left corner of the board.

\subsection{The Corner Case}

Suppose that all $k$ queens are located near the upper left corner of the board. If this is the case, then almost all squares attacked by these queens is on either a row, a column, or a long (i.e. downward sloping) diagonal with one of the queens. Furthermore, each such row, column and long diagonal containing a queen will lead to approximately $n$ attacked squares. Therefore, the optimal placement of $k$ queens will be the one that minimizes the number of such lines that are covered by them. To simplify the argument, we will need to turn this question on its head. Instead we ask, given $m$ rows, columns and long diagonals, what is the greatest number of queens that can be placed so that our queens attack only these rows, columns and long diagonals. In other words, so that each selected queen lies at the intersection of three of these lines.

\begin{figure}
\begin{center}
\hfill\includegraphics[scale=0.1]{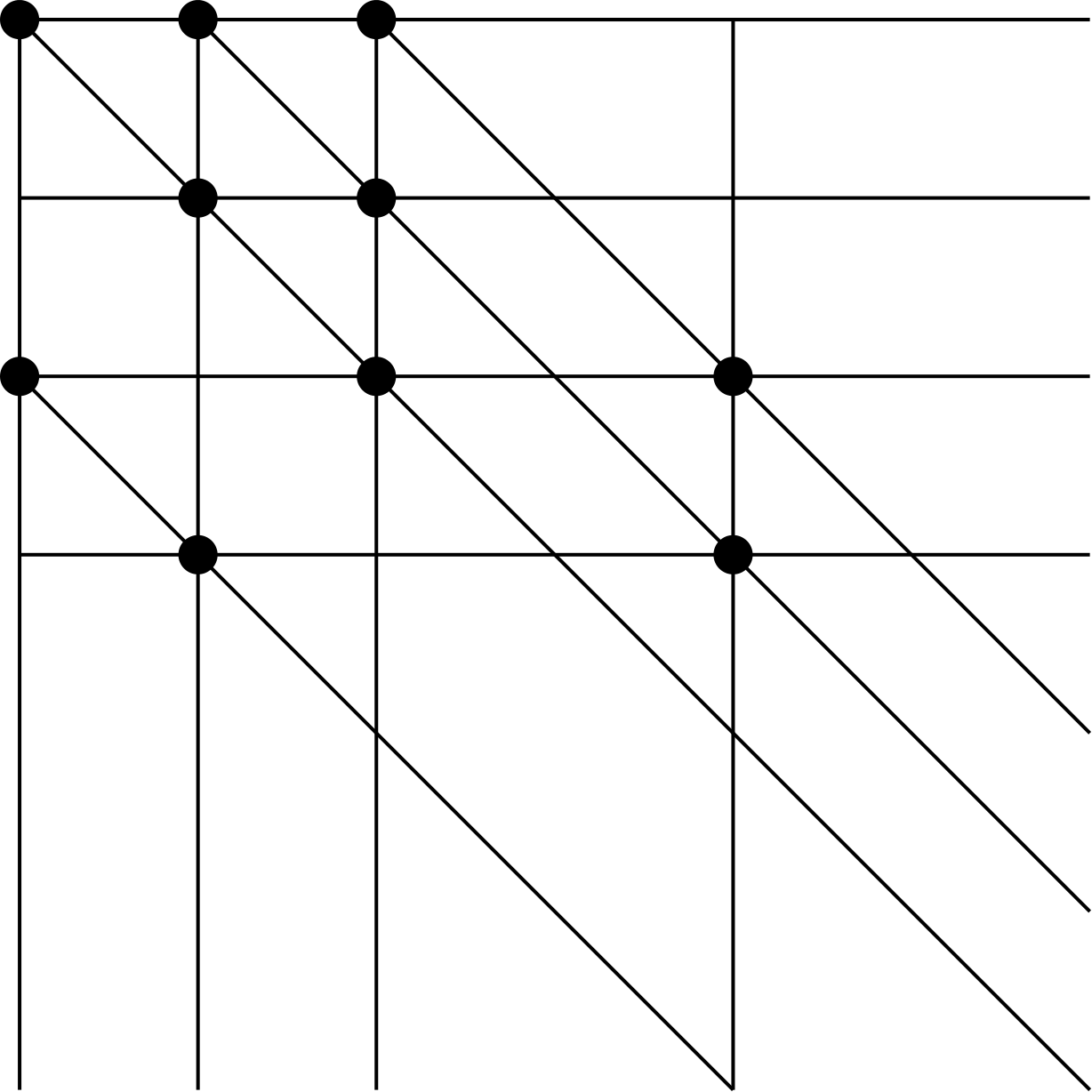}\hfill \includegraphics[scale=0.1]{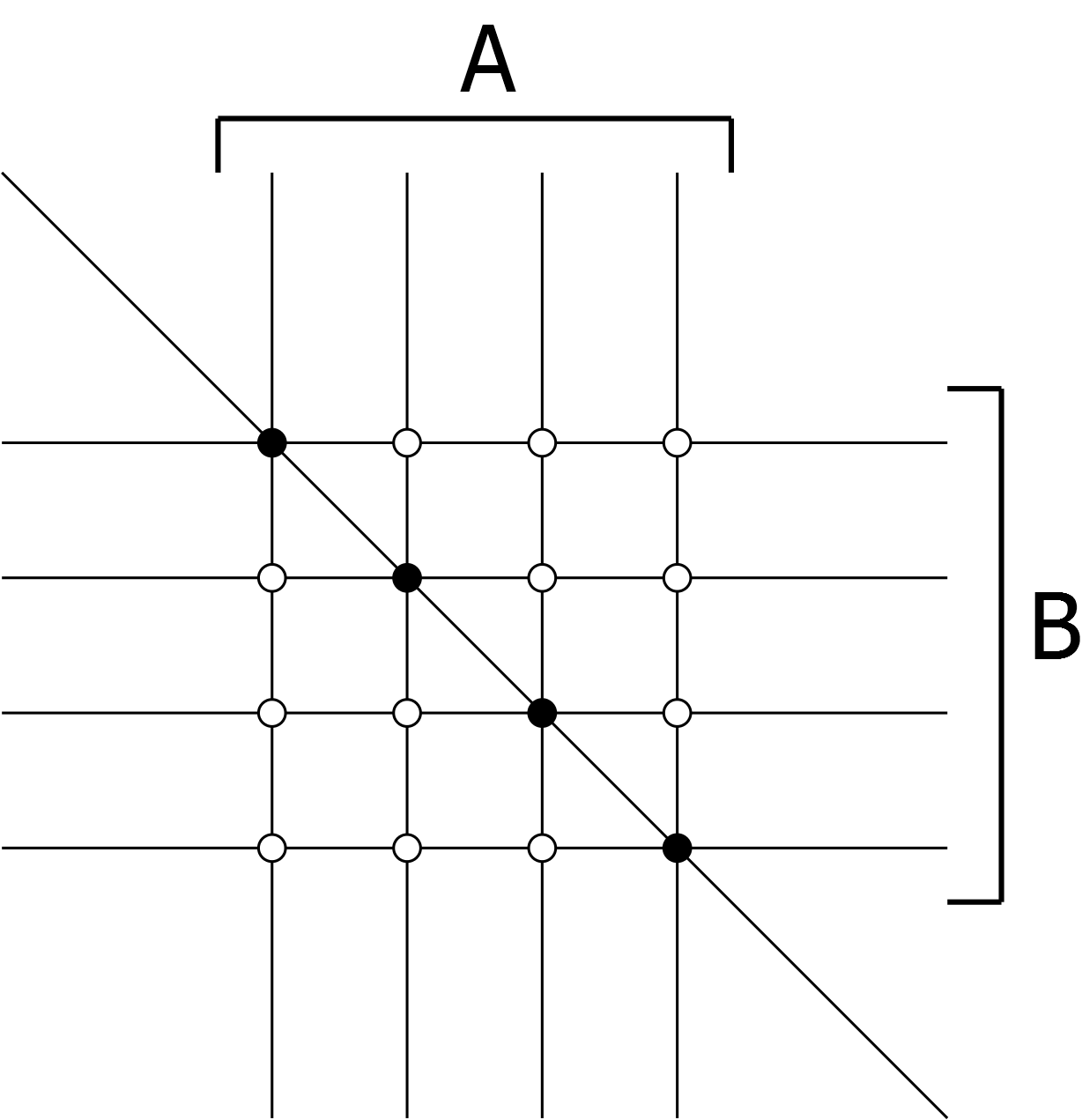}\hfill\phantom{.}
\caption{Left: The queens at the intersection of 4 rows, 4 columns and 4 long diagonals. Right: A diagonal intersecting the grid formed by rows and columns.}\label{cornerIntersectsFig}
\end{center}
\end{figure}

A convenient way to think about this is to first pick a set of $m$ rows, columns and long diagonals, and then look at the number of queens at three-way intersections as shown in Figure \ref{cornerIntersectsFig}. A further refinement of this idea comes by thinking about first fixing the $A$ columns and $B$ rows being used. This yields an $A\times B$ grid of possible queen locations. We then have to intersect this grid with $C$ (with $A+B+C=m$) diagonals and count intersections. This is complicated by the fact that the spacing between rows and columns in our ``grid'' does not need to be regular. A given diagonal could intersect as many as $\min(A,B)$ gridpoints. However, if we merely use this bound on our final answer, it will not be strong enough.

\begin{figure}
\begin{center}
\includegraphics[scale=0.1]{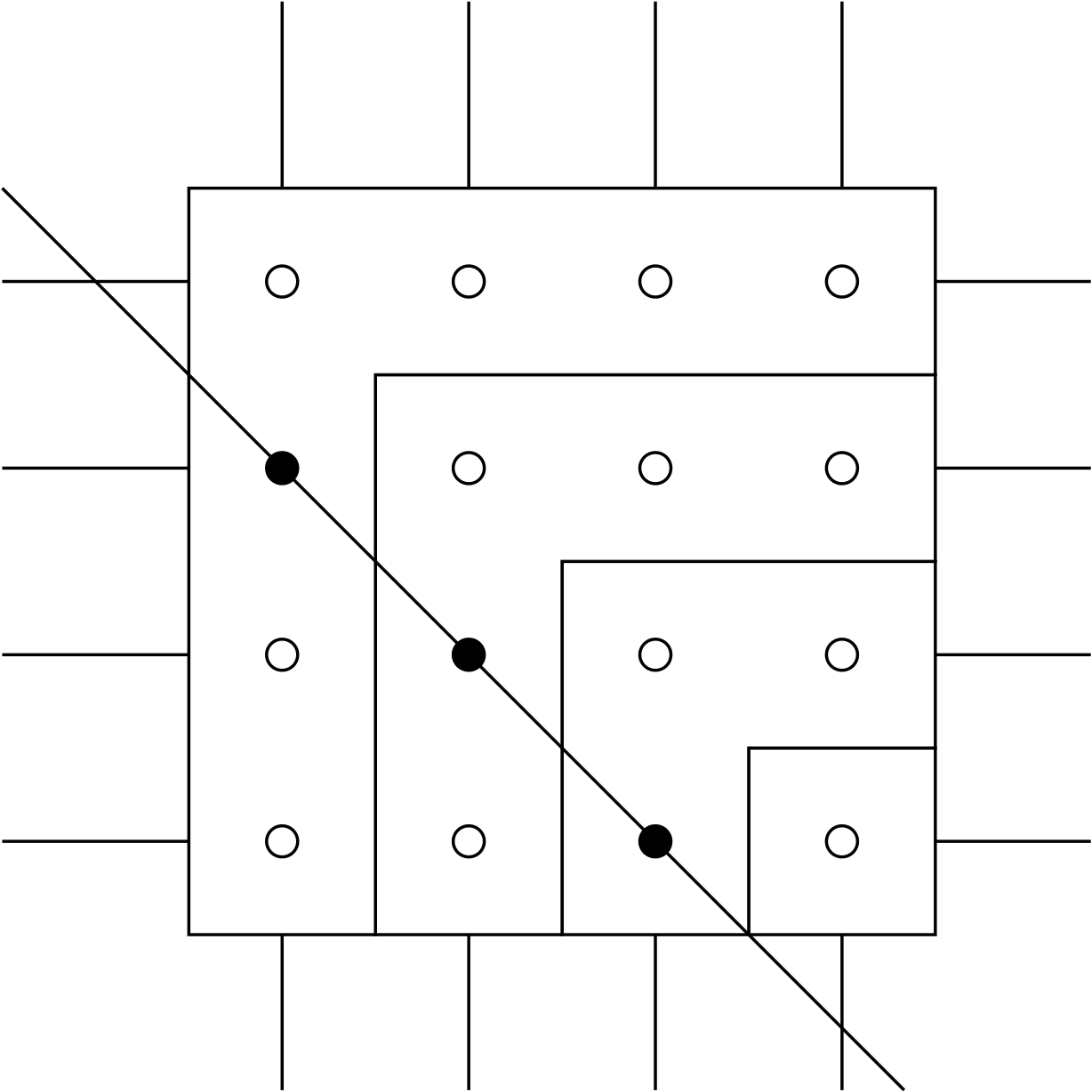}
\caption{Dividing the grid up into L's, and an intersecting diagonal.}\label{LsFig}
\end{center}
\end{figure}

The key idea to solve this issue is to divide the grid up into L's as shown in Figure \ref{LsFig}. One L contains the top row and rightmost column. The next contains the remaining points in the next to top row and next to right column, and so forth. The key thing to note is that given any two grid points in the same L, the sum of their $x$- and $y$- coordinates is not the same, and therefore any long diagonal intersects at most $1$ point in each L.

To complete our analysis, we count the number of points in each L that are on some diagonal. The $\ell^{th}$ L from the outside has a total of $A+B-2\ell+1$ grid points in it, so there are at most this many. On the other hand, since each long diagonal can cover at most one grid point in each L, the total number of points covered in each L is at most $C$. Therefore, the number of points at the intersection of a row, column, and diagonal is at most
$$
\sum_{\ell=1}^{\min(A,B)} \min(C,A+B+1-2\ell).
$$
With a little bit more work, it can be shown that this is largest when $A,B,$ and $C$ are as close to each other as possible, giving a bound of $k \leq \floor{\frac{m^2+3}{12}},$ which is exactly what we achieve with our hexagon construction.

\subsection{The General Argument}

The general case is somewhat more complicated. While one can still try similar methods to bound the number of rows, columns and diagonals attacked by our queens, we run into the difficulty that it is no longer the case that all relevant diagonals will be of the same length. In particular, showing that the queens lie on many diagonals will not necessarily imply that they can attack many different squares on these diagonals. In order to make the analysis work, we will need to show that at least some of these diagonals are long. In particular, we note that if you consider both diagonals through a single queen that the sum of the lengths is always reasonable. Along these lines, we will make use of the following Lemma:

\begin{figure}
\begin{center}
\hfill\includegraphics[scale=0.1]{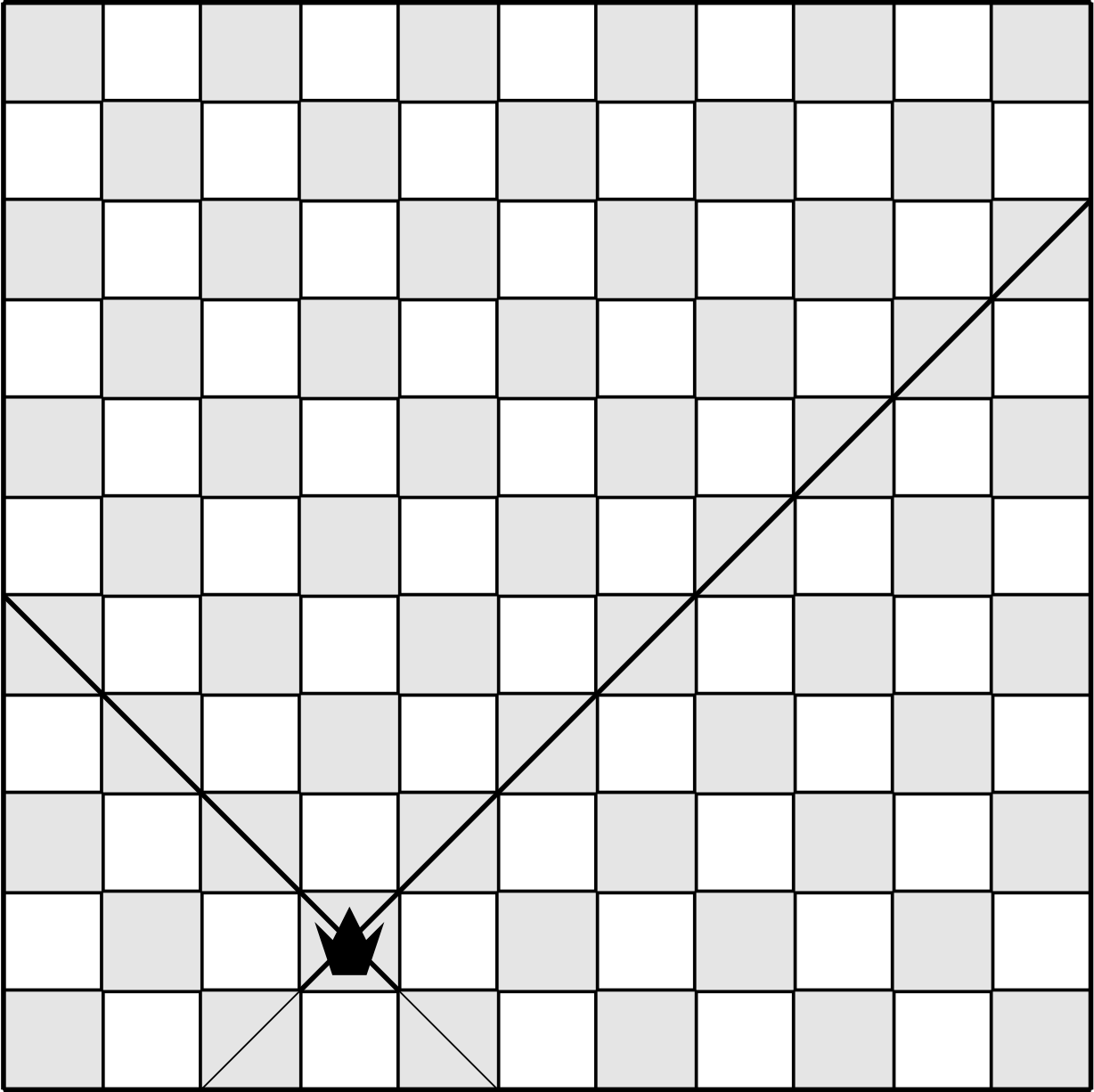} \hfill\includegraphics[scale=0.1]{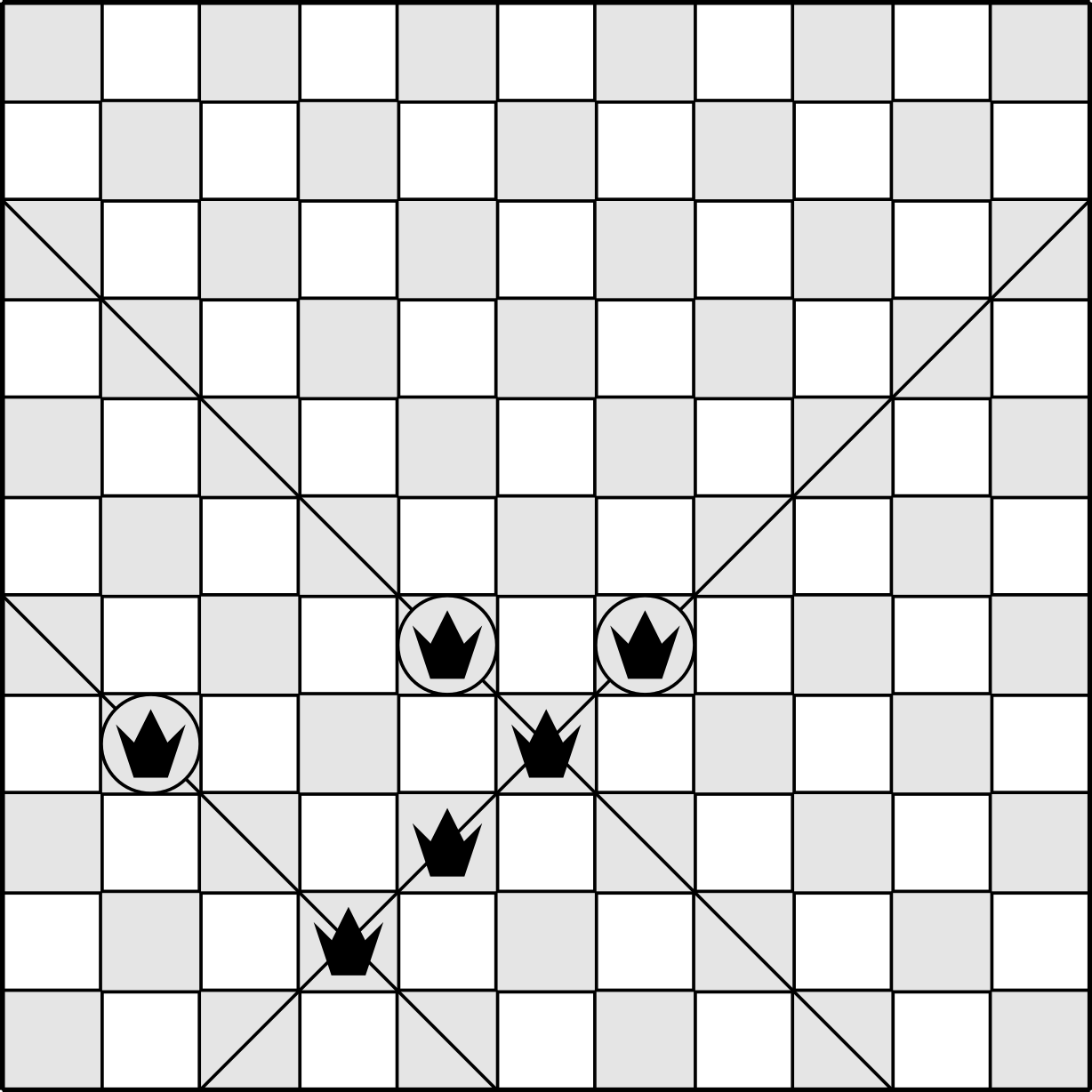}\hfill\phantom{.}
\caption{Left: The diagonals passing through a single queen. Note that the sum of the lengths of the darker segments is one more than the length of the board. Right: Three diagonals covering all queens, and three queens that do not share diagonals.}\label{onequeendiagonalsFig}
\end{center}
\end{figure}

\begin{lem}\label{twoDiagonalsLem}
Given any queen, the sum of the lengths of the diagonals passing through it is $n+1$ plus twice the distance to the nearest side of the board.
\end{lem}
\begin{proof}
Assume without loss of generality that the queen is as close to the bottom side of the board as any other. Then the diagonal going up-left from up reaches the left side of the board and the diagonal going up-right reaches the right side. Therefore, the sum of the lengths of these two half-diagonals is $n+1$ as they overlap in one square but otherwise intersect each column exactly once. The other halves of these diagonals each have length equal to the distance to the nearest edge of the board, so the sum of lengths is $n+1$ plus twice this distance.
\end{proof}

Therefore, if we look at the diagonals through a given queen, we know that the sum of their lengths must be at least $n$. Next, given our configuration of many queens, we will need to show that the sum of the lengths of all diagonals through them is significant. One way to do this would be by finding many queens so that no two of them shared a diagonal. This gives us a problem where we need to find the greatest number of such queens that do not share a diagonal. This can be thought of as a maximum matching problem between the diagonals of the two orientations. We find another characterization of this using K\"onig's Minimax Theorem. The final result we want is the following:

\begin{lem}\label{diagonalCoverLem}
Consider a set $S$ of queens on the board. Unless there is a collection of $C-1$ diagonals so that each queen in $S$ is on one of these diagonals, then there exists a set of $2C$ diagonals passing through queens in $S$ the sum of whose lengths is at least $nC$.
\end{lem}
\begin{proof}
Call a diagonal positive if it goes from down-left to up-right and negative otherwise. Consider the collection of all positive and negative diagonals passing through any of our queens. Consider two of them to be connected if there is a queen in $S$ on the intersection of those diagonals. Let $M$ be the maximum possible number of queens that we can select from $S$ so that no two share a common diagonal. By Lemma \ref{twoDiagonalsLem}, it suffices to show that $M\geq C$.

Notice that these connections between the diagonals, give them the structure of a bipartite graph, $G$, and that a collection of queens not sharing a diagonal is a matching on $G$. Therefore, $M$ is the size of a maximum matching of $G$. However, by K\"onig's Minimax Theorem the size of this maximum matching is the same as the size of a minimal vertex cover. In this case that means the minimum number of diagonals needed so that every queen in $S$ is on one of the chosen diagonals. Therefore unless all there is a collection of $C-1$ diagonals so that each queen in $S$ in one one of them, we will have $M\geq C$.
\end{proof}

We are now ready to prove Proposition \ref{mainProp}. The basic strategy will be the same as in the simple case where all queens were in a corner, except that we will need to count diagonals in both directions.

In particular, consider a set $S$ of $k$ queens. Let $A$ be the number of columns passing through queens in $S$, and let $B$ be the number of rows. Let $C$ be the smallest integer so that the queens in $S$ can be covered by $C$ diagonals. Let $M=A+B+C$. We note that by using the diagonals mentioned in Lemma \ref{diagonalCoverLem} that there exists a set of $O(M)$ rows, columns and diagonals all passing through queens in $S$ the sum of whose lengths is at least $nM$. It suffices to prove that $k\leq G(M)$.

Thus, we are left with the problem of finding the best way to select $A$ columns, $B$ rows and $C$ diagonals (now of either orientation), so that $A+B+C=M$ and so that the number of squares at the intersection of a selected row, a selected column, and a selected diagonal is as large as possible. In particular, we would like to show that if we fix a set of $A$ columns, $B$ rows, and $C$ diagonals with $A+B+C\leq M$, that the set $S$ of points lying on at least one chosen row, one chosen column, and one chosen diagonal satisfies $|S|\leq G(M)$.

We let the $x$-coordinates of the columns be $x_1<x_2<\cdots < x_A$ and let the $y$-coordinates of the rows be $y_1<y_2<\cdots<y_B$. We consider the grid $I:= \{(x_i,y_j):1\leq i\leq A,1\leq j \leq B\}$ defined by these rows and columns, and note that $S\subseteq I$.

We next partition the set $I$ into \emph{rings}, similarly to how we defined L's before. We begin with the set of points on the outermost edge in any direction. Next, we consider the points one set inward from that, and so on. Formally, we let the $\ell^{th}$ ring consist of points $(x_i,y_j)$ where $\min(i,j,A+1-i,B+1-j)=\ell$. Thus, consisting of the set of points $(x_\ell,y_\ell),(x_\ell,y_{\ell+1}),\ldots,(x_\ell,y_{B+1-\ell}),$ $(x_{\ell+1},y_{B+1-\ell}),\ldots,(x_{A+1-\ell},y_{B+1-\ell}),$ $(x_{A+1-\ell},y_{B-\ell}),\ldots,(x_{A+1-\ell},y_\ell),$\\ $(x_{A-\ell},y_\ell),\ldots,(x_{\ell+1},y_\ell).$

We note that each ring can be partitioned into two sequences on each of which $x+y$ is guaranteed to be strictly increasing. In particular, the above ring can be partitioned into
$(x_\ell,y_\ell),(x_\ell,y_{\ell+1}),\ldots,(x_\ell,y_{B+1-\ell}),$ $(x_{\ell+1},y_{B+1-\ell}),\ldots,(x_{A+1-\ell},y_{B+1-\ell}),$ and \\ $(x_{\ell+1},y_\ell),$ $(x_{\ell+2},y_\ell),\ldots,(x_{A+1-\ell},y_\ell)$ $,(x_{A+1-\ell}$ $,y_{\ell+1}),\ldots,(x_{A+1-\ell},y_{B-\ell}).$ This means that each negatively oriented line can only intersect each ring in at most $2$ points. Furthermore, an analogous argument says that each positively oriented line can only intersect each ring in at most $2$ points.

\begin{figure}
\begin{center}
\includegraphics[scale=0.1]{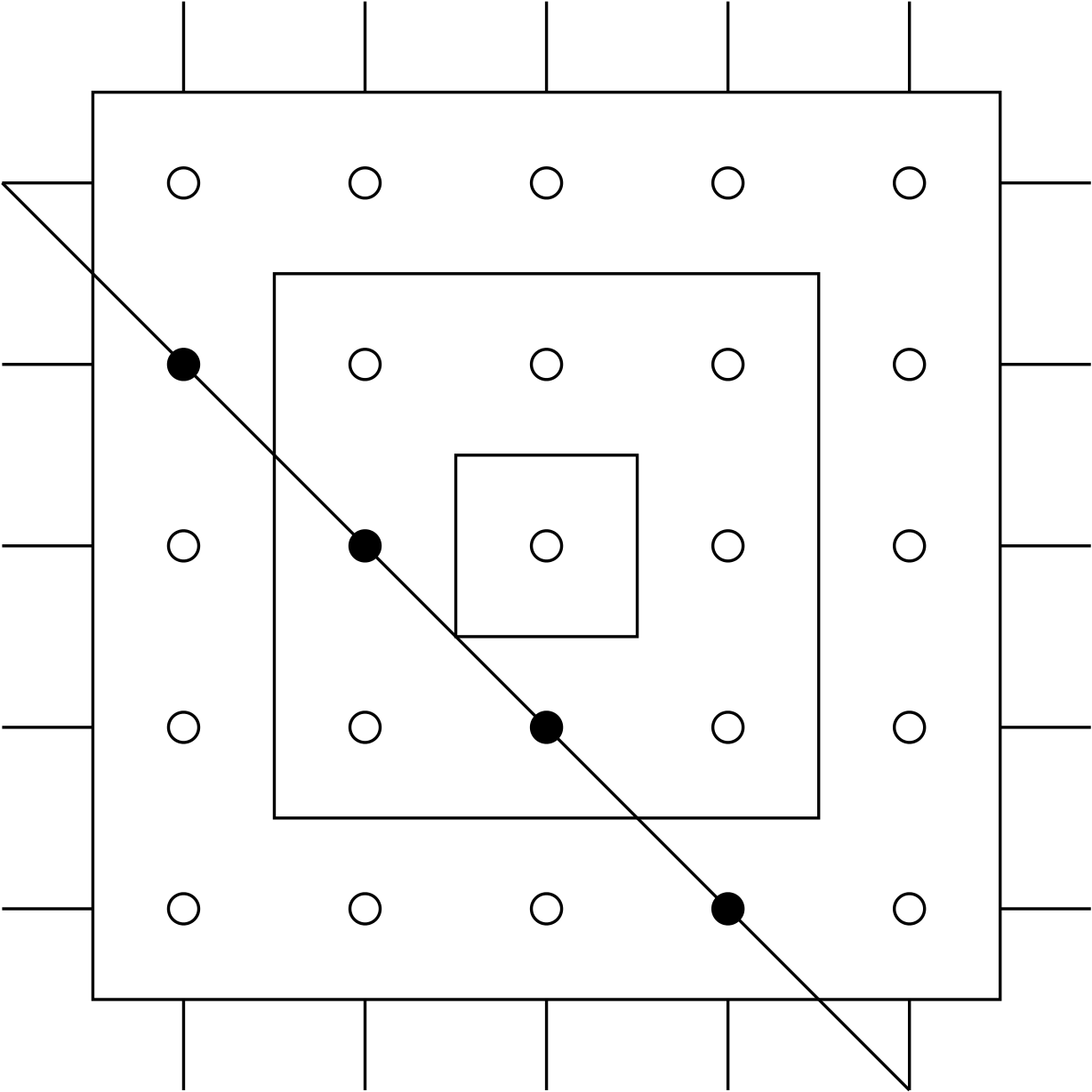}
\caption{The grid divided up into rings and one diagonal.}\label{ringsFig}
\end{center}
\end{figure}

Hence, we arrive at the following important result:
\begin{claim}
Any diagonal intersects each ring in ay most $2$ points.
\end{claim}

Therefore, the number of points in the $\ell^{th}$ ring that lie on one of our $C$ diagonals is at most $2C$.

Note that the number of points in the $\ell^{th}$ ring is $2A+2B+4-8\ell$ if $A,B\geq 2\ell$, $0$ if either is less than $2\ell-1$ and $\max(A,B)-2(\ell-1)$ if $\min(A,B)=2\ell-1$. Notice that in any case that this implies that the number of points in the $\ell^{th}$ ring is at most $\max(2A+2B+4-8\ell,0)$ unless $A=B=2\ell-1$, in which case there is one point (instead of $0$). The number of points of $S$ in the $\ell^{th}$ ring is also at most $2C$ by the above claim, so summing over $\ell$ we find that
\begin{equation}\label{boundEqn}
|S| \leq \sum_{\ell=1}^{\floor{\frac{A+B+2}{4}}} \min(2C,2A+2B+4-8\ell) +\delta_{A+B}= f(A+B,C)+\delta_{A+B}.
\end{equation}
Where $\delta_{A+B}$ is $1$ if $A+B\equiv 2\pmod{4}$ and $0$ otherwise.
Therefore, we have that $|S|$ is at most
$$
F(M) := \max_{A,C\in\Z_+,A+C=M} f(A,C)+\delta_A.
$$
We proceed to compute this value.

The biggest difficulty here is determining the values of $A$ and $C$ that achieve the maximum value of $f(A,C)+\delta_A$. In order to do this, we note that it is not hard to figure out what happens if we increase $A$ by $1$ and decrease $C$ by $1$ or visa versa. At the maximum value, neither of these operations can lead to an increase.

First, we note that as long as $C\leq A+2$ that
$$
f(A+1,C) = f(A,C) + 2\left(\floor{\frac{A+2}{4}} - \floor{\frac{A-C+2}{4}} \right).
$$
This is because increasing $A$ by $1$ increases each of the summands with $\floor{\frac{A-C+2}{4}} < \ell \leq \floor{\frac{A+2}{4}}$ by $2$, leaves the others unchanged, and possibly introduces a new summand of value $0$.

We also see that as long as $C\leq A+1$ that
$$
f(A,C+1) = f(A,C) + 2\floor{\frac{A-C+1}{4}}.
$$
This is because each term with $\ell \leq \floor{\frac{A-C+1}{4}}$ is increased by $2$.

Combining the above we find that for $C\leq A$
\begin{align*}
f(A+1,C-1) & = f(A,C-1) + 2\left(\floor{\frac{A+2}{4}} - \floor{\frac{A-C+3}{4}} \right) \\ & = f(A,C) + 2\left(\floor{\frac{A+2}{4}} - \floor{\frac{A-C+3}{4}} - \floor{\frac{A-C+2}{4}}\right).
\end{align*}
Similarly, we have that
\begin{align*}
f(A-1,C+1) & = f(A-1,C) + 2\floor{\frac{A-C}{4}} \\ & = f(A,C) + 2\left(\floor{\frac{A-C}{4}}+\floor{\frac{A-C+1}{4}}-\floor{\frac{A+1}{4}} \right).
\end{align*}

We claim that $F$ is optimized by the unique $A$ so that $2(M-A)+2 \geq A\geq 2(M-A),$ unless $M\equiv -2\pmod{12}$, in which case the optimum is attained when $A=2(M-A)-2$, or when $M\equiv 8\pmod{12}$, in which case $A=2(M-A)-1$. In particular, we show that for smaller $A$ increasing $A$ by $1$ and decreasing $C$ by $1$ doesn't decrease $f(A,C)+\delta_A$, while if $A$ is larger, then decreasing $A$ by $1$ and increasing $C$ by $1$ does this. This can be verified directly from the above by a somewhat tedious computation.

Hence, we have that $F$ is optimized when $C=\floor{M/3}$ unless $M$ is $10$ or $8$ modulo $12$, in which case it is $\floor{M/3}+1$. Therefore,
$$
F(M)=f(M-C,C)+\delta_{M-C}.
$$
For the appropriate value of $C$.

We now compute $f(M)$ by conditioning on its congruence class mod 12. It is not hard to show that
$$
F(M)=\begin{cases}\floor{\frac{M^2}{12}} & \textrm{ if }M\equiv 0,1,5,7,11 \pmod{12} \\ & \\ \floor{\frac{M^2}{12}}+1 & \textrm{ if }M\equiv 2,3,4,6,8,9,10 \pmod{12} \end{cases}.
$$

We therefore, have that any set of more than $F(M)$ queens will have $O(M)$ lines passing through these queens containing a total of at least $(M+1)n$ squares. We note that this is already sufficient to establish our lower bound except in the cases where $M$ is congruent to $2,4,8$ or $10$ modulo $12$, in which case it might be possible to fit in one extra queen without attacking too many squares. In order to handle these cases, we will need a separate argument to show that in the unusual case that we can pick $A$ columns, $B$ rows and $C$ diagonals with $A+B+C=M$ so that each of $G(M)+1$ queens lie at the intersection of three of these lines, that we will still attack too many squares.

We begin by noting that in each of these cases, there is a unique set of values of $A$ and $C$ that optimize $F$. This can again be verified directly by a somewhat tedious examination of cases.

Therefore, in any of the cases any configuration of $F(M)$ queens that do not lie on $O(M)$ diagonals with total length at least $(M+1)n$ must have some $A$ columns, $B$ rows, and $C$ diagonals so that $A+B+C=M$, $A+B$ is $8s+2,8s+3,8s+5$, or $8s+6$ depending on case with $A$ and $B$ differing by at most 1 so that the $F(M)$ points are covered by the $A$ columns, covered by the $B$ rows and covered by the $C$ diagonals. Call these $C$ diagonals the \emph{designated diagonals}. Furthermore, these covers must be minimal in terms of the number of rows/columns/diagonals used. Furthermore, if we take the $x$-coordinates of the columns to be $x_1<\ldots<x_A$ and the $y$-coordinates of the rows to be $y_1<\ldots<y_B$, then for each $\ell<(2A+2B+4)/8$, the $\ell^{th}$ ring as defined above must contain exactly $\min(2C,2A+2B+4-8\ell)$ of our points.

Note that in each of the above cases, there is an $\ell$ so that $2C=2A+2B+4-8\ell$ (in particular $\ell=s+1$). We call the corresponding ring the \emph{critical ring}. Our analysis from this point on will depend crucially on this critical ring. The first thing to note is that in order for our bound to be tight, each of the $C$ designated diagonals must intersect the critical ring in exactly two points, and correspondingly each point in the critical ring must be on exactly one designated diagonal. Note also that each designated diagonal must intersect each ring further out than the critical ring in exactly two places. We will show that in each of these four cases, that this is impossible unless the sum of the lengths of all the diagonals containing queens in the critical ring is substantially larger than $Cn$.

For notational purposes, we call diagonals positively or negatively oriented depending on the sign of their slope. Consider the point $(x_{\ell},y_{\ell})$ at the corner of the critical ring. We note that the designated diagonal passing through it must be positively oriented for otherwise it could not intersect the critical ring again. As this diagonal must pass through $\ell-1$ other points to the lower left of this one each with distinct $x$ and $y$ coordinates, it must pass through $(x_1,y_1)$. This implies that $x_{\ell}-x_1 = y_\ell-y_1$. Applying this argument to the other corners, we find that $y_\ell-y_1=x_A-x_{A+1-\ell}=y_B-y_{B+1-\ell}$. We call this common difference $D$ as shown in Figure \ref{ddefFig}.

\begin{figure}
\begin{center}
\hfill\includegraphics[scale=0.1]{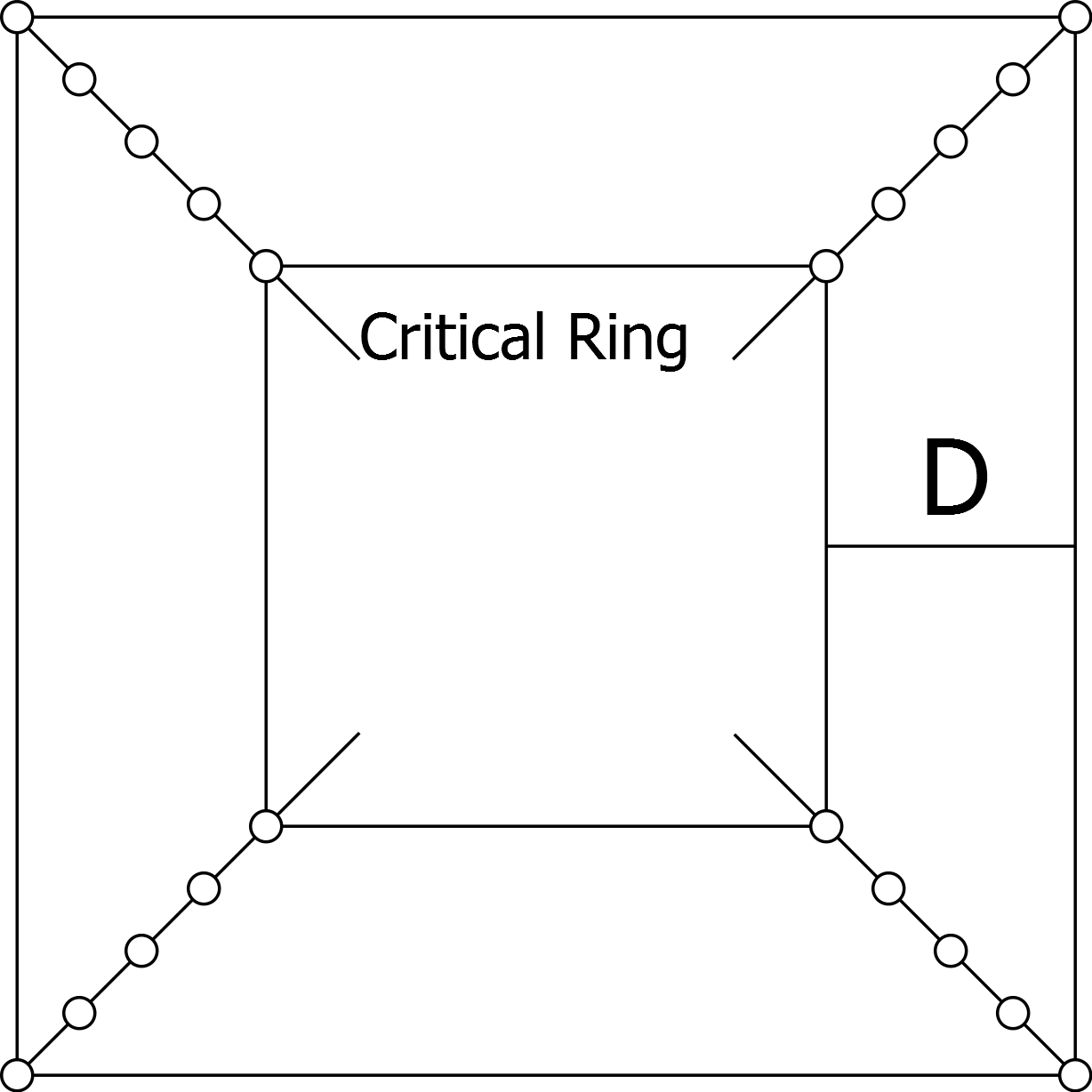}\hfill\includegraphics[scale=0.1]{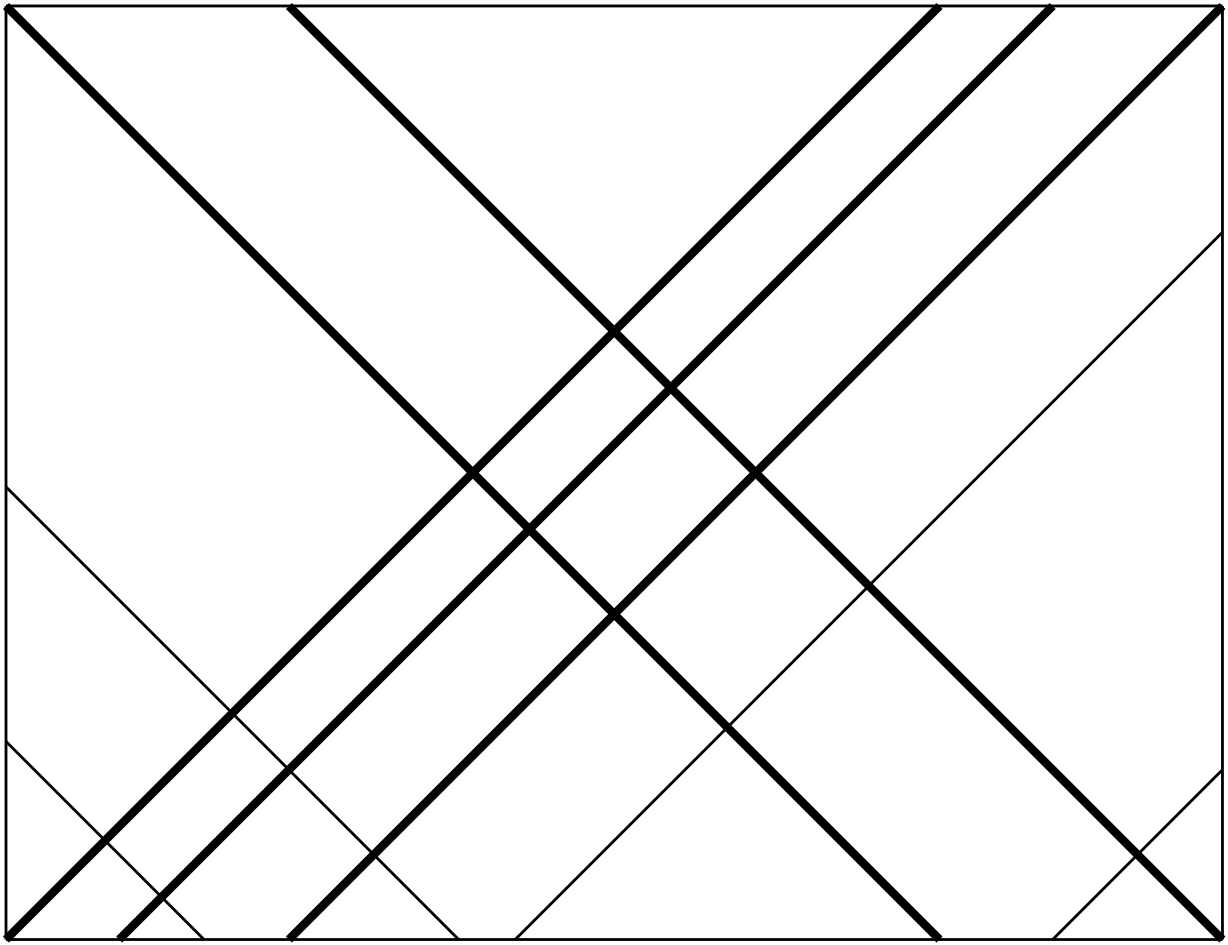}\hfill\phantom{.}
\caption{Left: The grid divided up into rings and one diagonal. Right: The difficulty with a non-square critical ring. Note that the light diagonals all intersect the top/bottom once and one of the sides once. Thus, the number of points on the top/bottom, must exceed the number of the sides by at least $4$.}\label{ddefFig}
\end{center}
\end{figure}

We now consider the critical diagonals passing through the critical ring. We begin by claiming that the critical ring is square, namely that $x_{A-\ell+1}-x_\ell=y_{B-\ell+1}-y_\ell$. If this were not the case, the critical ring would have a long side and a short side. The critical diagonals passing through the corners must each also pass through a point on one of the two long sides. Furthermore, no critical diagonal can connect two points on short sides. Thus, in this case the number of queens on the long sides of the critical ring must exceed the number of queens on the short sides by at least 4, and therefore, we must have $|A-B|> 1$, which is a contradiction.

Given that the critical ring is square, the critical diagonals through the corners must pass through opposite corners. The critical diagonals through the sides of the critical ring must each intersect both a horizontal side and a vertical side. Therefore, the number of queens on each side of the critical diagonal must be the same. Therefore, we must have that $A=B$. This excludes the cases where $M$ is congruent to $4$ or $8$ modulo 12.

For the case where $M=2$, it is trivially noted that one cannot have a single queen on no diagonals. For $M=10$, the larger value of $G(M)$ means that there is nothing to prove (this is the special case of the $9$ queen configuration). Therefore, we may assume that $M\geq 12$ and $M\equiv 2,10\pmod{12}$.

Unfortunately, in these cases, it is possible to find arrangements of $G(M)+1$ queens that lie on only $A$ columns, $B$ rows and $C$ diagonals with $A+B+C=M$. For example, the arrangement in Figure \ref{17Qs} for $M=14$ places $17$ queens on $5$ rows, $5$ columns, and $4$ diagonals. However, any instantiation of this arrangement will necessarily attack far more than $14n$ squares.

\begin{figure}
\begin{center}
\includegraphics[scale=0.4]{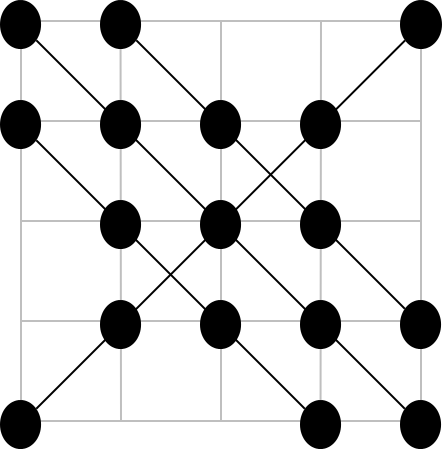}
\end{center}
\caption{ An arrangement of seventeen queens to lie on five rows, five columns and four diagonals.}\label{17Qs}
\end{figure}

However, we claim that in all such cases the sum of the lengths of the diagonals through the points in the critical ring is still at least $(C+1)n$, and thus that we can find $O(M)$ rows, columns and diagonals through selected queens with total length at least $(M+1)n$. To prove this, we will need the following Lemma:

\begin{lem}\label{critDiagLem}
The sum of the lengths of diagonals through points in the critical ring is at least $Cn+2CD$ plus half the sum of the lengths of diagonals that intersect exactly one point in the critical ring.
\end{lem}
\begin{proof}
The sum of the lengths of these diagonals equals half the sum of twice the lengths of these diagonals. This equals half of the sum over points $p$ in the critical ring of the sum of the lengths of the diagonals through $p$ plus the sum of the lengths of diagonals that intersect the critical ring in exactly one point. The former term is at least $n+2D$ for each $p$ by Lemma \ref{twoDiagonalsLem}. As there are exactly $2C$ queens on the critical ring, this completes the proof.
\end{proof}

Now, we only have to show that either $D$ is large or that the sum of the lengths of diagonals intersecting the critical ring exactly once is large. We note that such diagonals must exist. In particular, there must be at least four \emph{skew diagonals}, that is positive diagonal through the upper left or lower right corner, or the negative diagonal through the lower left or upper right corner as shown in Figure \ref{skewFig}.

\begin{figure}
\begin{center}
\hfill\includegraphics[scale=0.1]{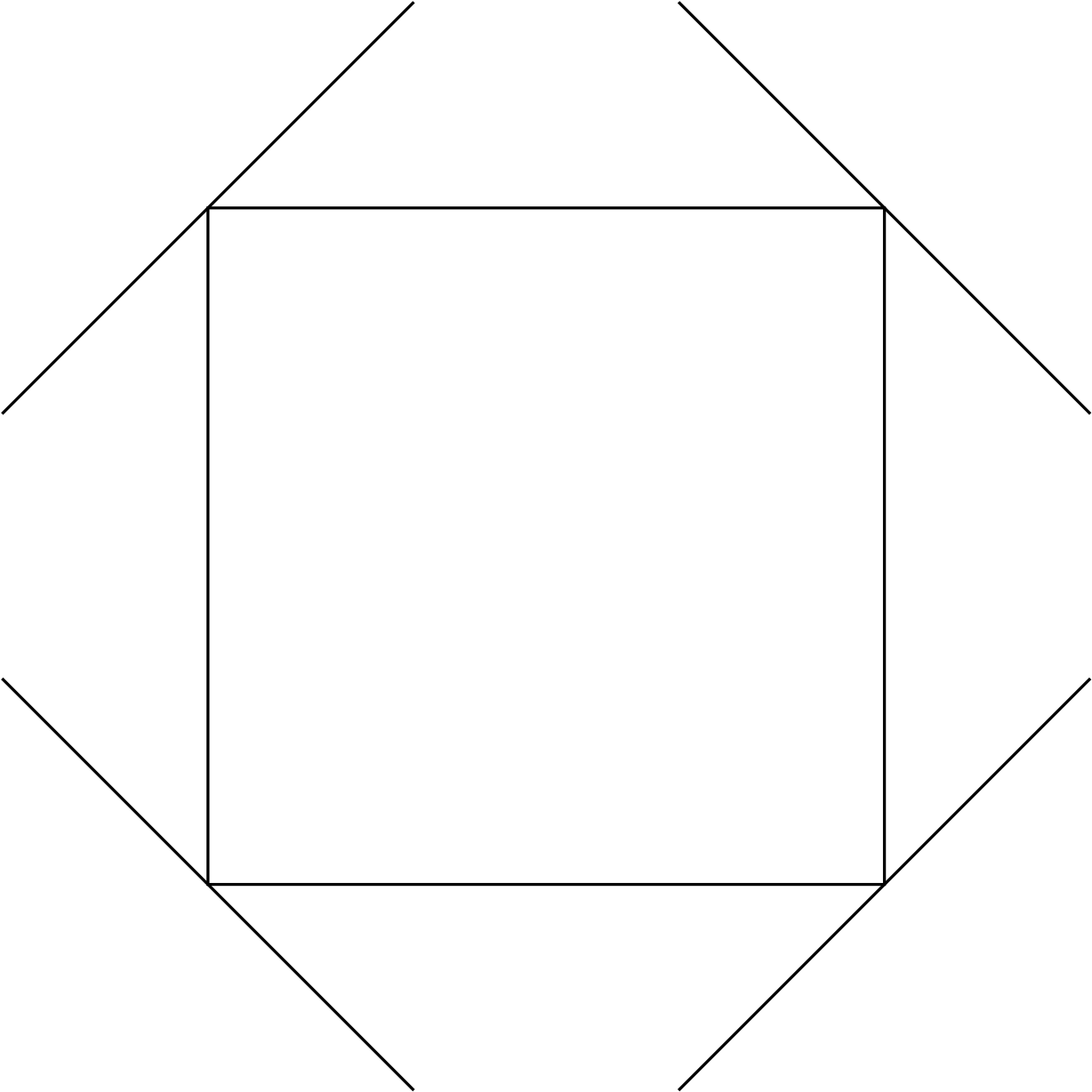}\hfill\includegraphics[scale=0.1]{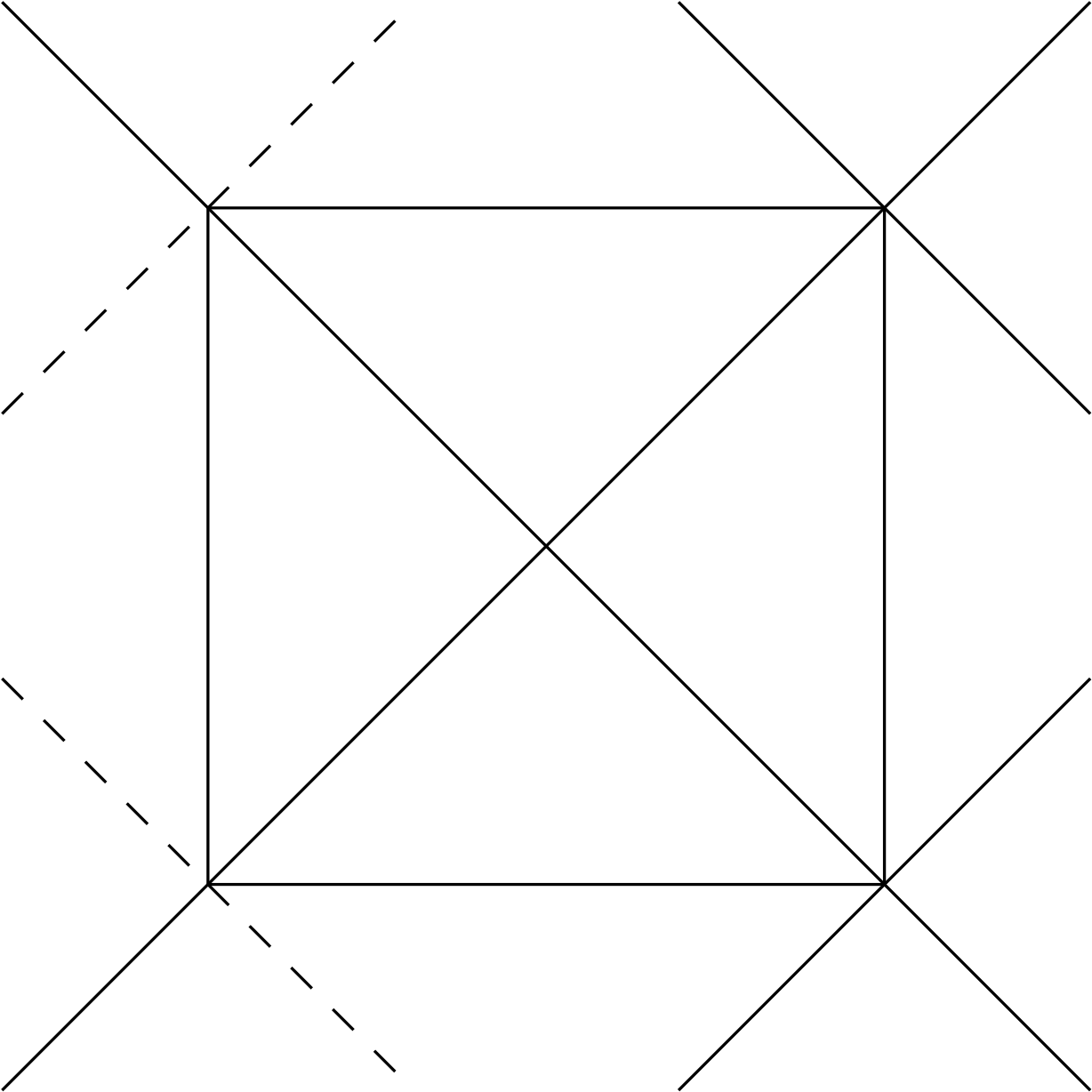}\hfill\phantom{.}
\caption{Left: The critical ring and skew diagonals. Right: The skew diagonals after the leftmost ones are moved right.}\label{skewFig}
\end{center}
\end{figure}

We note that the sum of the lengths of the skew diagonals is actually fairly long. In particular, if we let $L=x_{A+1-\ell}-x_\ell = y_{B+1-\ell}-y_\ell$ be the length of the side of the critical ring, we note that by moving the lower left and upper left skew diagonals $L$ units to the right, we are left with all four diagonals passing through the upper right and lower right corners. The sums of the lengths of these diagonals is now at least $2n+4D$. However, moving the other diagonals left may have reduced their lengths by as much as $L$ each. Therefore, we have the following claim:

\begin{claim}\label{totalLenClaim}
The sum of the lengths of the skew diagonals is at least $2n+4D-2L$. Furthermore, it is at least this plus half the sum of the lengths of the non-skew diagonals passing through only one point on the critical ring.
\end{claim}

Combining this with Lemma \ref{critDiagLem}, we have that the sum of the lengths of diagonals through points in the critical ring is at least
$$
(C+1)n+2(C+1)D-L.
$$
To complete our proof, we merely have to show that $L\leq 2(C+1)D$.

In other words, we are claiming that the distance between the two sides of the square is small compared to $D$. We begin with a much simpler statement that the medial row/column is not too far from at least one of the sides.

\begin{claim}
$\min(x_{A+1-\ell}-x_{(A+1)/2},x_{(A+1)/2}-x_\ell)\leq 2D$.
\end{claim}
\begin{proof}
We begin by noting that a diagonal through the left side of the critical ring will intersect again on the top side if it is positively oriented and on the bottom side if it is negatively oriented. Since the number of points on the left side is the same as the number on the right, we determine that the number of positively oriented critical diagonals passing through the top of the critical ring is equal to the number of positively oriented critical diagonals passing through the bottom of the critical ring. Since the number of points on each side of the critical ring is odd, we may assume without loss of generality that the number of critical diagonals passing through either top or bottom of the critical ring is greater than the corresponding number of negative diagonals. Therefore, there must be an $\ell < i < A+1-\ell$ so that there are positively oriented critical diagonals passing through both $(x_i,y_\ell)$ and $(x_i,y_{B+1-\ell})$.

The critical diagonal through $(x_i,y_\ell)$ will also pass through $(x_j,y_1)$ with $j \leq i-\ell+1$, and the critical diagonal passing through $(x_i,y_{B+1-\ell})$ also passes through $(x_{j'},y_B)$ with $j'\geq i+\ell-1$. Note that $j'-j \geq 2\ell-2$ and that $x_{j'}-x_i=x_i-x_j = D$. Notice that it is definitely the case that $j' \geq (A+1)/2 \geq j$, thus if either $\ell\geq j$ or $j'\geq A+\ell-1$, the claim will be proved. The only way for neither of the above to hold would be if $M\equiv 10\pmod{12}$ and $i=(A+1)/2$, $j=\ell+1$, and $j'=A-\ell$.

In this case, we consider the critical diagonals through $(x_j,y_\ell)$ and $(x_{j'},y_\ell)$. We note that we cannot have the former be negatively oriented and the latter positively oriented because then the former diagonal would pass through $(x_j+D,y_1)=(x_i,y_1)$, and the latter would pass through $(x_{j'}-D,y_1)=(x_i,y_1)$. Therefore, we may assume without loss of generality that the diagonal through $(x_j,y_\ell)$ is positively oriented. However, this would imply that it passes through $(x_t,y_1)$ for $t\leq \ell$ and thus $x_j-x_\ell\leq D$, and therefore $x_i-x_\ell\leq 2D.$

This completes the proof of the claim.
\end{proof}

\begin{figure}
\begin{center}
\hfill\includegraphics[scale=0.1]{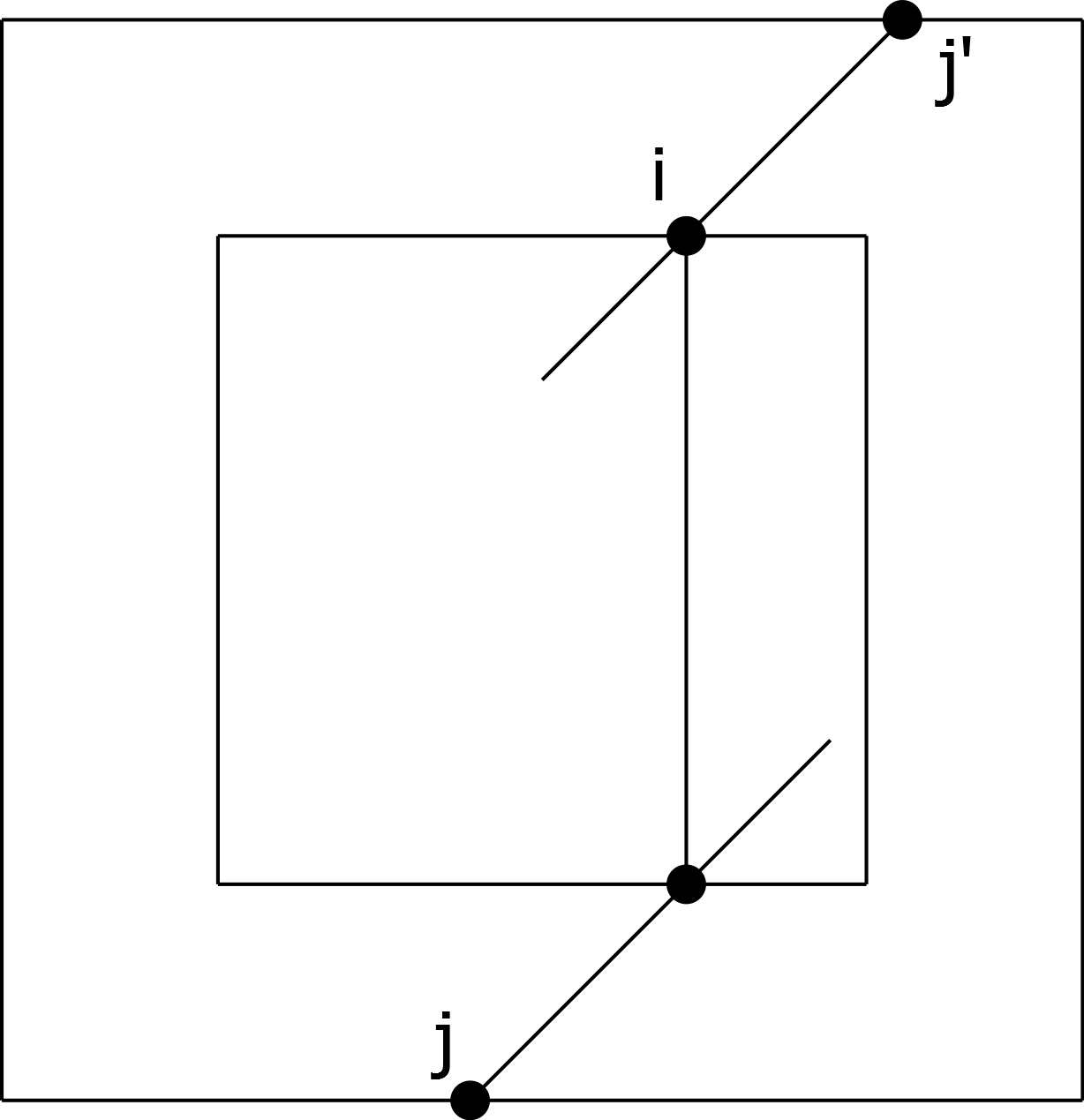}\hfill\includegraphics[scale=0.1]{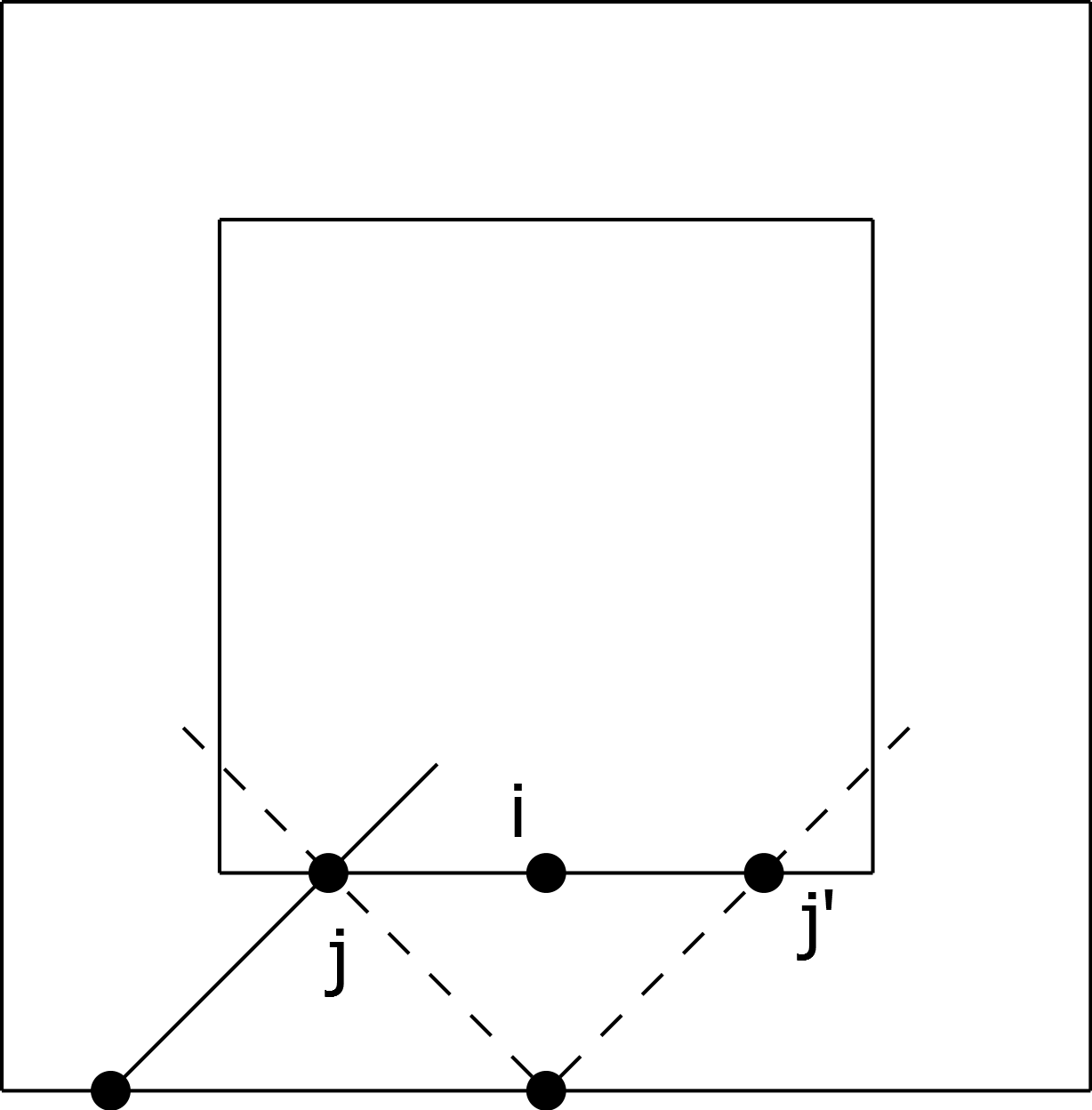}\hfill\phantom{.}
\caption{Left: The critical ring and two positively oriented critical diagonals passing through $(x_i,y_\ell)$ and $(x_i,y_{B+1-\ell})$. Right: A positively oriented diagonal through $(x_j,y_\ell)$.}\label{stackedpositivesFig}
\end{center}
\end{figure}

By the claim, we can assume without loss of generality that $x_{(A+1)/2}-x_\ell\leq 2D$. Similarly, to the above claim we also get that $\min(y_{B+1-\ell}-y_{(B+1)/2},y_{(B+1)/2}-y_\ell)\leq 2D$. Thus, we may also assume that $y_{(B+1)/2}-y_\ell\leq 2D.$

We note that there are a large number of points on the critical ring very close to the bottom left corner. In particular, we have $(x_\ell,y_i)$ for $\ell\leq i\leq (B+1)/2$ and $(x_i,y_\ell)$ for $\ell \leq i \leq (A+1)/2$. We are going to consider the positively oriented diagonals through these points.

On the one hand, suppose that such a diagonal passes through $(x_{A+1-\ell},y_j)$ with $j\leq (B+1)/2$ or $(x_j,y_{B+1-\ell})$ with $j\leq (A+1)/2$. Without loss of generality, the line passes through $(x_i,y_\ell)$ and $(x_{A+1-\ell},y_j)$ with $\ell \leq i \leq (A+1)/2$ and $\ell \leq j \leq (B+1)/2$. Then we have that
$$
L = x_{A+1-\ell}-x_\ell = (x_{A+1-\ell}-x_i)+(x_i-x_\ell) = (y_j-y_\ell)+(x_i-x_\ell) \leq 2D+2D = 4D.
$$
So in this case $L\leq 4D$, and we are done by Claim \ref{totalLenClaim}.

In the other case, where no such diagonal exists, the only other points of the critical ring that these diagonals can pass through are $(x_{A+1-\ell},y_j)$ for $(B+1)/2< j \leq B+1-\ell$ and $(x_j,y_{B+1-\ell})$ for $(A+1)/2 \leq j \leq A+1-\ell$. Note that there are two fewer such points, than points of the form we are starting with, and therefore at least two of these (non-skew) diagonals intersect the critical ring in at least one point. Furthermore, it is easy to see that each of these diagonals have length at least $L-2D$. Therefore by Claim \ref{totalLenClaim}, the sum of the lengths of the diagonals through queens in the critical ring in this case must be at least
$$
(C+1)n + 2(C+1)D - L + (L-2D) > (C+1)n.
$$

\begin{figure}
\begin{center}
\hfill\includegraphics[scale=0.1]{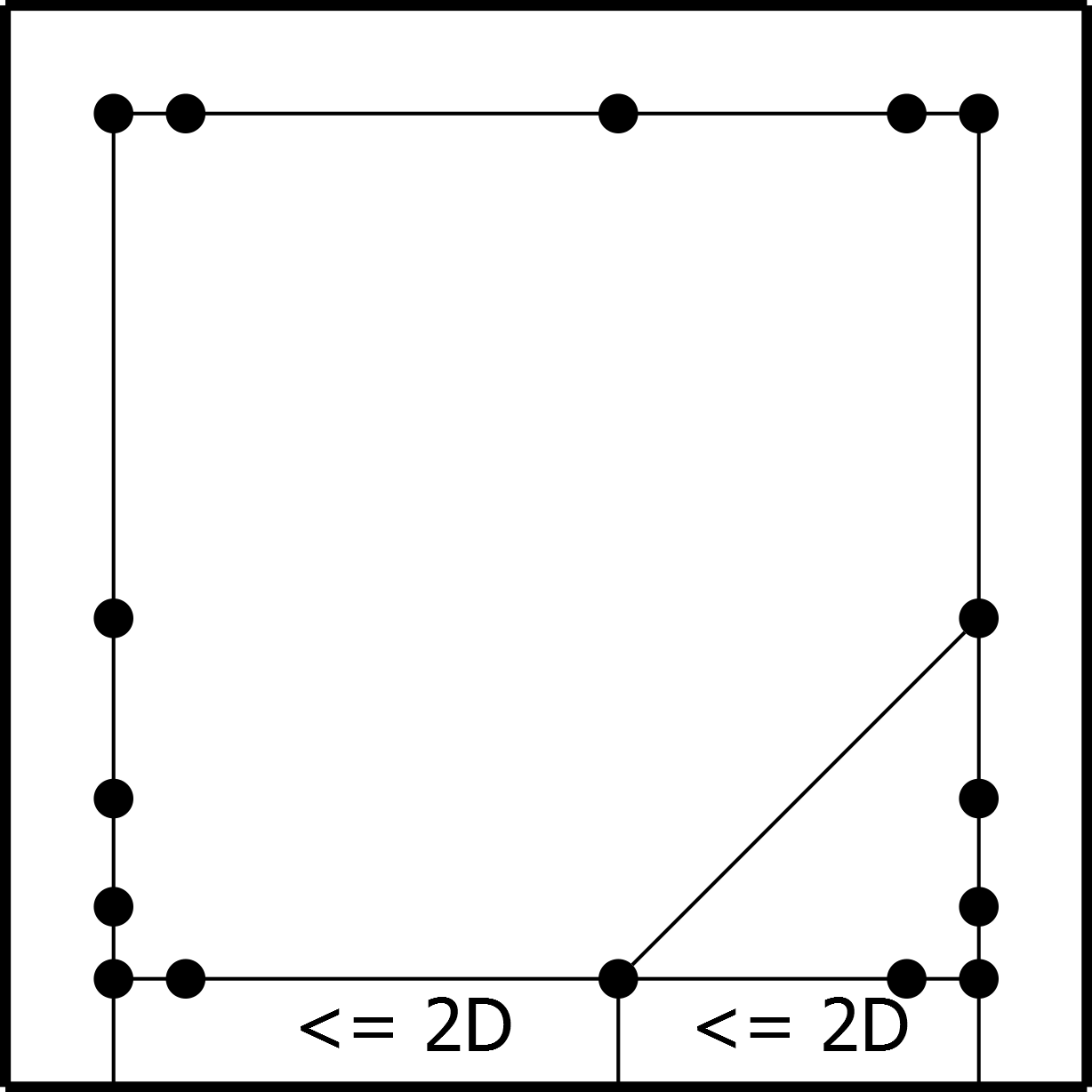}\hfill\includegraphics[scale=0.1]{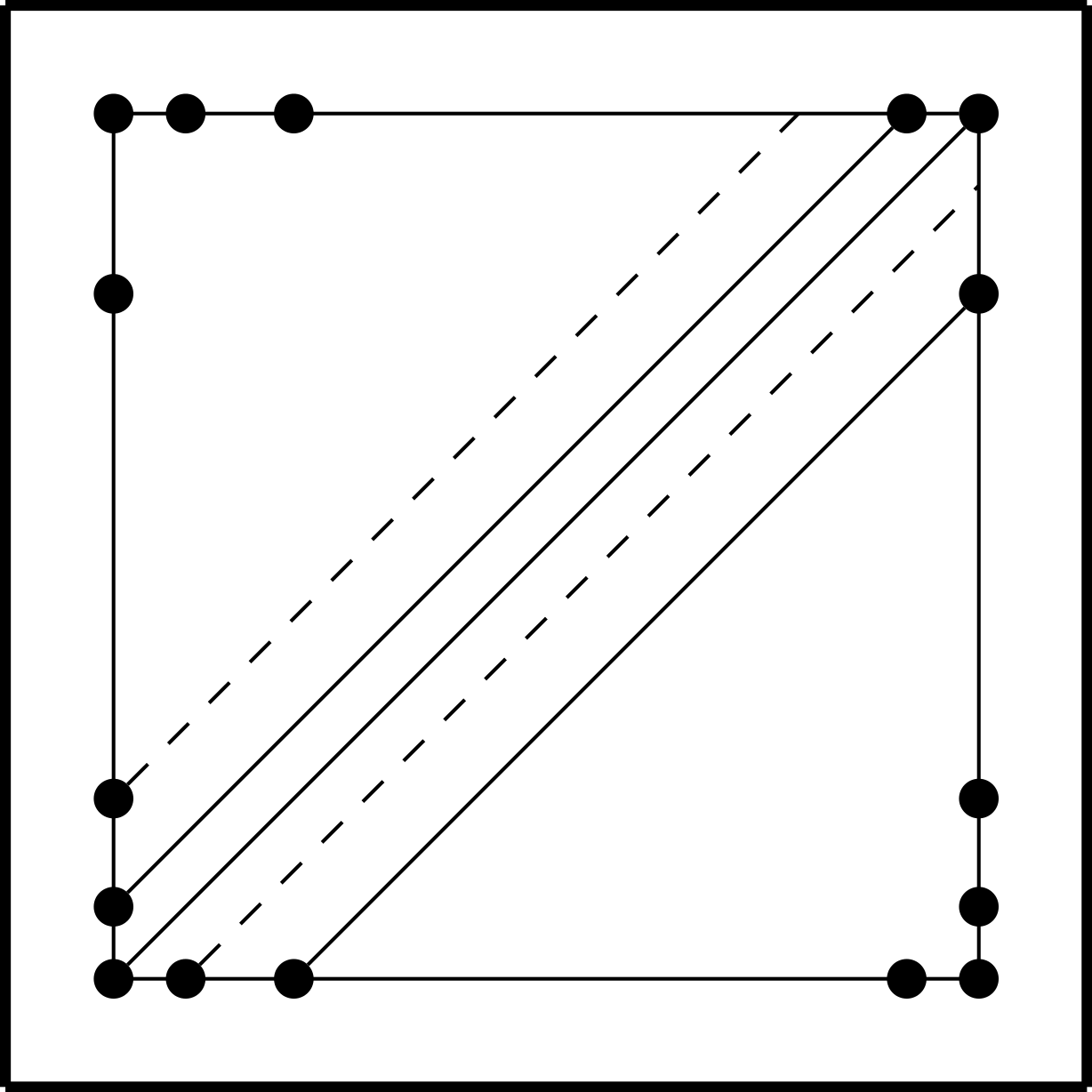}\hfill\phantom{.}
\caption{Two possibilities. Left: A diagonal connecting a vertex close to the bottom left to one close to the bottom right, showing that $L$ is small. Right: $L$ is much larger than $D$ and we have additional diagonals intersecting the critical ring in only one point.}\label{missingdiagsFig}
\end{center}
\end{figure}

Thus, things work out in either case, completing our proof.

\section*{Acknowledgements}

I would like to thank Ron Graham and Donald Knuth for their helpful discussions about this problem and useful comments about the presentation of this paper. This work was supported in part by the NSF-CAREER grant award number 1553288.

\end{document}